\documentclass[12pt,a4paper]{amsart}
\usepackage{amsmath,amsfonts,amssymb,amsthm,amscd}
\usepackage{dsfont,mathrsfs,enumerate}
\usepackage{curves,epic,eepic}
\usepackage{float}

\usepackage[colorlinks=true, linkcolor=blue, citecolor=blue, pagebackref]{hyperref}

\usepackage{url}
\usepackage{cancel}
\usepackage{pdflscape}

\usepackage[pdftex]{graphicx}
\usepackage{psfrag}
\usepackage{epic}
\usepackage{eepic}
\usepackage[matrix,arrow,curve]{xy}
\input{epsf}
\usepackage{tikz}
\usetikzlibrary{calc}
\usetikzlibrary{intersections}
\usetikzlibrary{through}
\usetikzlibrary{backgrounds}

\newtheorem{thm}{Theorem}[section]
\newtheorem{lemma}[thm]{Lemma}
\newtheorem{prop}[thm]{Proposition}
\newtheorem{coro}[thm]{Corollary}

\newtheorem{classify}[thm]{Classification}
\theoremstyle{definition}
\newtheorem{defi}[thm]{Definition}
\newtheorem{rem}[thm]{Remark}

\def\Z{\mathds Z}
\def\Q{\mathds Q}
\def\R{\mathds R}

\def\phi{\varphi}
\def\<{{\langle}}
\def\>{{\rangle}}

\newcommand{\vol}[1]{\mathrm{vol}(#1)}
\newcommand{\area}[1]{\mathrm{area}\left(#1\right)}
\newcommand{\interior}[1]{\mathrm{int}(#1)}

\newcommand{\ceil}[1]{\left\lceil #1 \right\rceil}
\newcommand{\ehr}[2]{\mathrm{ehr}_{#1}\left(#2\right)}
\newcommand{\denom}[1]{\mathrm{denom}(#1)}
\newcommand{\conv}[1]{\mathrm{convhull}\left(#1\right)}
\newcommand{\inthull}[1]{\mathrm{inthull}\left(#1\right)}
\newcommand{\ldist}[2]{\mathrm{ldist}_{#1}(#2)}

\begin{document}
	
\title[Quasi-period collapse in half-integral polygons]{Quasi-period collapse in half-integral polygons}
\author[Martin Bohnert]{Martin Bohnert}
\address{Mathematisches Institut, Universit\"at T\"ubingen,
Auf der Morgenstelle 10, 72076 T\"ubingen, Germany}
\email{martin.bohnert@uni-tuebingen.de}
\maketitle
\thispagestyle{empty}
	
\begin{abstract}
A half-integral polygon with quasi-period collapse behaves similarly to a lattice polygon in the sense that the number of lattice points in its integer dilates can be calculated as values of a polynomial, its Ehrhart polynomial. As a main result, we classify the Ehrhart polynomials of all half-integral non-lattice polygons with quasi-period collapse. In particular, we obtain that for any positive integer $i$, the polynomial $\frac{4i+5}{2}t^2+\frac{2i+7}{2}t+1\in \Q[t]$ is an Ehrhart polynomial of a rational polygon, which was an open question for $i>1$.
	
We also study some extreme cases in detail. In particular, we show that up to affine unimodular equivalence there exist exactly $30$ half-integral non-lattice polygons with quasi-periodic collapse with exactly one interior lattice point, which are the dual polygons of the $30$ LDP polygons of Gorenstein index $2$. Furthermore, we classify all half-integral polygons with quasi-period collapse with at most $6$ interior lattice points or with $i\geq 1$ interior lattice points and the maximum possible number $2i+7$ of boundary lattice points.
\end{abstract}
		
\section{Introduction}

Rational non-lattice polytopes, whose number of lattice points in their integer dilates can be calculated as the values of a polynomial as for lattice polytopes, first appeared as random single examples, e.g., the half-integral polytope with the vertices $(0,0,0)$, $(1,0,0)$, $(0,1,0)$, $(1,1,0)$ and $(\frac{1}{2},0,\frac{1}{2})$ in Stanley's \textit{Enumerative Combinatorics} \cite[4.6.10]{Sta12}. In \cite{DLMA04} some Gelfand-Tsetlin polytopes, a class of polytopes relevant to representation theory, are seen as the first infinite family of such \textit{pseudo-integral} non-lattice polytopes. Systematic constructions in \cite{MAW05} led to examples for any given denominator of the rational polytope in any dimension greater than $1$ and to a better understanding of the $2$ dimensional case. An explanation of the more general phenomena of quasi-period collapse by certain piecewise affine unimodular transformations between the rational polytope and a lattice polytope was proposed in \cite{HMA08}. For the special case of duals of LDP polygons, \cite{KW18} explained quasi-period collapse with methods of algebraic geometry motivated by mirror symmetry. The classification of Ehrhart quasi-polynomials of half-integral polygons was started in \cite{Her10}, and in \cite{MAM17} the classification of Ehrhart polynomials of pseudo-integral polygons was started, leaving open whether there are Ehrhart polynomials of the form
\begin{align*}
\frac{4i+k-2}{2}t^2+\frac{2i+k}{2}t+1\in \Q[t]
\end{align*}
for $i\geq 1, k\geq 7$ and $(i,k)\neq (1,7)$. Our main theorem will show in particular that there are half-integral polygons with quasi-period collapse and such an Ehrhart polynomial with $i>1, k=7$.

Why are we focusing on half-integral polygons here? In a sense, they lead to the simplest polytopes with quasi-period collapse, since according to \cite[2.1.]{MAW05} there is no quasi-period collapse in dimension $1$. Their denominator is only two, so half-integral polygons with quasi-period collapse are automatically pseudo-integral. Additionally, we will have an easy criterion to detect quasi-periodic collapse for them. We should emphasize that half-integral polygons are relevant in several areas of combinatorics. For example, they naturally arise when studying lattice 3-polytopes of lattice width $2$, because then we have a half-integral polygon as a natural 'midpolygon'. This happens, for instance, when classifying maximal lattice 3-polytopes without interior lattice points in \cite{AKW17} or when studying the Fine interior of lattice 3-polytopes of lattice width $2$ in \cite{Boh24b}.

We will now begin to introduce our notation and provide the necessary background of Ehrhart theory for our topic. For more information on Ehrhart theory, see \cite{Sta12}, \cite{BR15}, or \cite{HNP12}.

Let be $n\in \Z_{\geq 1}$ and $P\subseteq \R^n$ be a full-dimensional \textit{rational polytope}, i.e., the convex hull of finitely many points of $\Q^n\subseteq \R^n$, whose affine hull is $\R^n$. We denote convex hulls by $\conv{\cdot,\dotsc,\cdot}$. If $P$ can even be described as the convex hull of finitely many points of the lattice $\Z^n \subseteq \R^n$, then we call $P$ a \textit{lattice polytope}. For each rational polygon $P \subseteq \R^n$, the smallest $d\in \Z_{\geq 1}$ such that $dP:=\{d\cdot x\in \R^n \mid x\in P\}$ is a lattice polytope is called the \textit{denomiator of $P$} and is denoted by $\denom{P}$. We are particularly interested in the case of polygons, i.e. $n=2$, and especially in the case $n=2, \denom{P}=2$, where we call $P$ a \textit{half-integral polygon}.

To study the interaction between the rational polytope $P\subseteq \R^n$ and the lattice $\Z^n$, we are interested in the \textit{number of lattice points} $l(P):=|P\cap \Z^n|$, the \textit{number of interior lattice points} $i(P):=|\interior{P}\cap \Z^n|$ and the \textit{number of boundary lattice points} $b(P):=|\partial P\cap \Z^n|$ of $P$, where $\interior{P}$ is the topological interior and  $\partial P$ is the topological boundary of $P\subseteq \R^n$. If we denote a polygon by $P_{d,(i,b)}$, then it is a rational polygon with denominator $d$, $i$ interior and $b$ boundary lattice points. We study rational polygons only up to \textit{affine unimodular equivalence}, i.e. up to automorphisms of our lattice. Thus, two rational polygons $P\subseteq \R^n$ and $P'\subseteq \R^n$ are considered equivalent if and only if there exists an \textit{affine unimodular transformation}
\begin{align*}
T_{A,b}\colon \R^n\to \R^n, x\mapsto Ax+b, \qquad A\in \mathbf{GL}(n,\Z), b\in \Z^n
\end{align*}
with $T_{A,b}(P)=P'$. The numbers $l(P), i(P)$ and $b(P)$ are invariant under such affine unimodular transformations. Since $|\det(A)|=1$, the volume of $P$, denoted by $\vol{P}$ or $\area{P}$ for $n=2$, is also invariant.

\textit{Ehrhart theory} gives a connection between these invariants via a \textit{quasi-polynomial} of degree $n$. A quasi-polynomial of degree $n$ is a generalized polynomial 
\begin{align*}
q(t)=\sum_{k=0}^{n}q_k(t)t^k,
\end{align*}
where the coefficient functions $q_k(t)$ are periodic functions with integral periods, rational values on integers, and $q_n\neq 0$. The least common multiple of the periods of $q_0,\dotsc,q_n$ is called the \textit{quasi-period} of $q$. A quasi-polynomial is just a polynomial if its quasi-period is $1$. The main result of Ehrhart theory is as follows: 

\begin{thm}[\protect{\cite{Ehr62}, \cite[3.23]{BR15}}]\label{Ehrhart_theorem}
Let $P\subseteq \R^n$ be a rational polytope. 
Then there exists a quasi-polynomial of degree $n$ called the \textit{Ehrhart quasi-polynomial of $P$} 
\begin{align*}
\ehr{P}{t}=\sum_{k=0}^{n}e_k(t)t^k,
\end{align*}
with quasi-period dividing $\denom{P}$, $e_n(t)=\vol{P}$ and $\ehr{P}{k}=l(kP)$ for all positive integers $k$.
\end{thm}

We call the vector with the periods of the coefficient functions $e_0,\dotsc, e_n$ of $\ehr{P}{t}$ as coefficients the \textit{period sequence of $P$} and the quasi-period of $\ehr{P}{t}$ the \textit{quasi-period of $P$}. Furthermore, we say that $P$ has \textit{quasi-period collapse} if its quasi-period is strictly less than $\denom{P}$. If the quasi-period of $P$ is $1$, i.e., if the Ehrhart quasi-polynomial is actually a polynomial, then we call $P$ a \textit{pseudo-integral polytope}. In particular, the half-integral polygons with quasi-period collapse are exactly the pseudo-integral ones.

An important property of the Ehrhart quasi-polynomial is the following \textit{Ehrhart-Macdonald reciprocity}, which connects $\ehr{P}{k}$ and $i(kP)$.

\begin{thm}[\protect{\cite[(4.6).]{Mac71}, \cite[4.1]{BR15}}]\label{Ehrhart-Macdonald}
Let $P\subseteq \R^n$ be a rational polygon with the Ehrhart quasi-polynomial $\ehr{P}{t}$. Then we have for $k\in \Z_{\geq 1}$ that
\begin{align*}
i(kP)=(-1)^n\ehr{P}{-k}.
\end{align*}
\end{thm}

For lattice polygons, the relation between $\area{P}$, $i(P)$ and $b(P)$ is particularly simple and can be described by the following formula of Pick. 

\begin{thm}[\protect{\cite{Pic99}, \cite[2.8]{BR15}}]\label{Pick}
Let $P\subseteq \R^2$ be a lattice polygon. Then $P$ obeys Pick's formula
\begin{align*}
\area{P}=i(P)+\frac{b(P)}{2}-1.
\end{align*}
\end{thm}

Ehrhart-Macdonald reciprocity and Pick's theorem \ref{Pick} allow us to calculate the Ehrhart polynomial of a lattice polygon $P$ from $\area{P}$ and $b(P)$. By \ref{Ehrhart_theorem} we have $\ehr{P}{t}=\area{P}t^2+e_1t+e_0$ and by Ehrhart-Macdonald reciprocity we have
\begin{align*}
e_1=\frac{1}{2}\left(\ehr{P}{1}-\ehr{P}{-1}\right)=\frac{1}{2}(l(P)-i(P))=\frac{b(P)}{2}.
\end{align*}
From Pick's theorem we obtain
\begin{align*}
e_0=\ehr{P}{1}-\area{P}-\frac{b(P)}{2}=1
\end{align*}
and thus 
\begin{align*}
\ehr{P}{t}=\area{P}t^2+\frac{b(P)}{2}t+1.
\end{align*}

To fully classify the Ehrhart polynomials for lattice polygons, another important connection between $b(P)$ and $i(P)$ is required. The following inequality of Scott was originally proved in \cite{Sco76} and since then several times, e.g., in \cite{HS09} or recently in \cite{Boh23}.

\begin{thm}[\cite{Sco76}]\label{Scott_inequality}
Let $P\subseteq \R^2$ be a lattice polygon with $i(P)>0$. Then either $b(P)\leq 2i(P)+6$ or $P$ is affine unimodular equivalent to the lattice triangle $\conv{(0,0), (3,0), (0,3)}$ and we have $i(P)=1, b(P)=9$.
\end{thm}

For all pairs $(i,b)\in \Z_{\geq 1}^2$ with $3\leq b\leq 2i+6$ there exists a lattice polygon $P_{1,(i,b)}$ with $i(P_{1,(i,b)})=i$ and $b(P_{1,(i,b)})=b$. 

For example, for $b=3$, we can use
\begin{align*}
P_{1,(i,3)}:=\conv{(0,1), (1,-1), (i+1,0)}
\end{align*}
and for	$i>0, 3<b\leq 2i+6$, we can use
\begin{align*}
P_{1,(i,b)}:=\conv{(0,-1),(b-4,-1), (i+1,0),(0,1)}.
\end{align*}
We also have for $b\geq 3$ the triangle
\begin{align*}
P_{1,(0,b)}:=\conv{(0,0),(b-2,0), (0,1)}
\end{align*}
with $i(P_{1,(0,b)})=0, b(P_{1,(0,b)})=b$.

\begin{figure}[H]
	\begin{tikzpicture}[x=0.4cm,y=0.4cm]
		\draw[step=2.0,black,thin,xshift=0cm,yshift=0cm] (-5,3) grid (29.9,-3);
		
		\draw [line width=1pt,color=black] (-4,2)-- (-2,-2);
		\draw [line width=1pt,color=black] (-2,-2)-- (0,0);
		\draw [line width=1pt,color=black] (0,0)-- (-4,2);
		\draw [line width=1pt,color=black] (-2,-2)-- (2,0);
		\draw [line width=1pt,color=black] (2,0)-- (-4,2);
		\draw [line width=1pt,color=black] (-2,-2)-- (4,0);
		\draw [line width=1pt,color=black] (4,0)-- (-4,2);
		
		\draw [fill=black] (-4,2) circle (2.5pt);
		\draw [fill=black] (-2,-2) circle (2.5pt);
		\draw [fill=black] (0,0) circle (2.5pt);
		\draw [fill=black] (2,0) circle (2.5pt);
		\draw [fill=black] (4,0) circle (2.5pt);
		\draw [fill=black] (-2,0) circle (2.5pt);
		
		\draw[color=black] (3,2) node[fill=white] {$P_{1,(i,3)}$ };

		\draw [line width=1pt,color=black] (8,2)-- (8,-2);
		\draw [line width=1pt,color=black] (8,-2)-- (12,0);
		\draw [line width=1pt,color=black] (12,0)-- (8,2);
		\draw [line width=2pt,color=black] (8,-2)-- (16,-2);
		\draw [line width=1pt,color=black] (12,0)-- (10,-2);
		\draw [line width=1pt,color=black] (12,0)-- (12,-2);
		\draw [line width=1pt,color=black] (12,0)-- (14,-2);
		\draw [line width=1pt,color=black] (12,0)-- (16,-2);

		\draw [fill=black] (8,2) circle (2.5pt);
		\draw [fill=black] (8,-2) circle (2.5pt);
		\draw [fill=black] (8,0) circle (2.5pt);
		\draw [fill=black] (10,0) circle (2.5pt);
		\draw [fill=black] (12,0) circle (2.5pt);
		\draw [fill=black] (10,-2) circle (2.5pt);
		\draw [fill=black] (12,-2) circle (2.5pt);
		\draw [fill=black] (14,-2) circle (2.5pt);
		\draw [fill=black] (16,-2) circle (2.5pt);
		
		\draw[color=black] (12,2) node[fill=white] {$P_{1,(1,b)}$ };

		\draw [line width=1pt,color=black] (20,0)-- (20,-2);
		\draw [line width=1pt,color=black] (20,0)-- (22,-2);
		\draw [line width=1pt,color=black] (20,0)-- (24,-2);
		\draw [line width=1pt,color=black] (20,0)-- (26,-2);
		\draw [line width=2pt,color=black] (20,-2)-- (26,-2);

		\draw [fill=black] (20,0) circle (2.5pt);
		\draw [fill=black] (20,-2) circle (2.5pt);
		\draw [fill=black] (22,-2) circle (2.5pt);
		\draw [fill=black] (24,-2) circle (2.5pt);
		\draw [fill=black] (26,-2) circle (2.5pt);
		
		\draw[color=black] (25,0) node[fill=white] {$P_{1,(0,b)}$ };
		
	\end{tikzpicture}
	\caption{Lattice polygons, which show how to produce families, which are affine unimodular equivalent to $P_{1,(i,3)}$, $P_{1,(1,b)}$ and $P_{1,(0,b)}$. }
\end{figure}
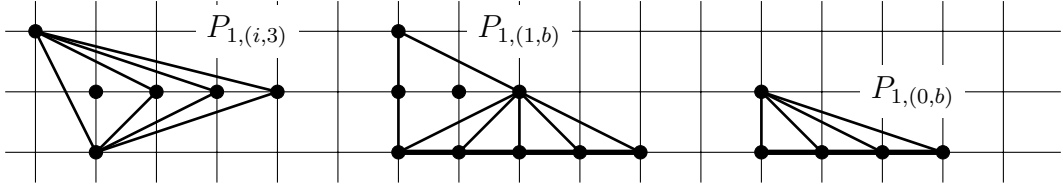

Thus, Scott's inequality gives a complete classification of Ehrhart polynomials of lattice polygons, as formulated in the following corollary. 

\begin{coro}[\protect{\cite[3.3.34]{HNP12}}]
The polynomial $e_2t^2+e_1t+e_0\in \Q[t]$ is an Ehrhart polynomial of a lattice polygon if and only if $e_0=1$ and there exists a pair
\begin{align*}
(i,b)\in \{0\}\times \Z_{\geq 3} \ \dot{\cup} \  \{(x,y)\in \Z_{\geq 0}^2 \mid x>0, \ 3\leq y\leq 2x+6 \} \ \dot{\cup} \ \{(1,9)\} 
\end{align*}
with $e_2=i+\frac{b}{2}-1$ and $e_1=\frac{b}{2}$.
\end{coro}

The main result of this paper is the following complete classification of Ehrhart polynomials of pseudo-integral polygons with denominator $2$.

\begin{thm}\label{main_result}
The polynomial $e_2t^2+e_1t+e_0\in \Q[t]$ is an Ehrhart polynomial of a pseudo-integral polygon with denominator $2$ if and only if $e_0=1$ and there exists a pair
\begin{align*}
(i,b)\in \{(0,3)\} \ \dot{\cup} \  \{(x,y)\in \Z_{\geq 0}^2 \mid x>0, \ 2\leq y\leq 2x+7 \} 
\end{align*}
with $e_2=i+\frac{b}{2}-1$ and $e_1=\frac{b}{2}$.
\end{thm}

To see the difference between the Ehrhart polynomials of lattice polygons and of the pseudo-integral polygons of denominator $2$, we can visualize possible pairs $(i(P),b(P))\in \Z^2$, as done in figure \ref{possible_pairs}.

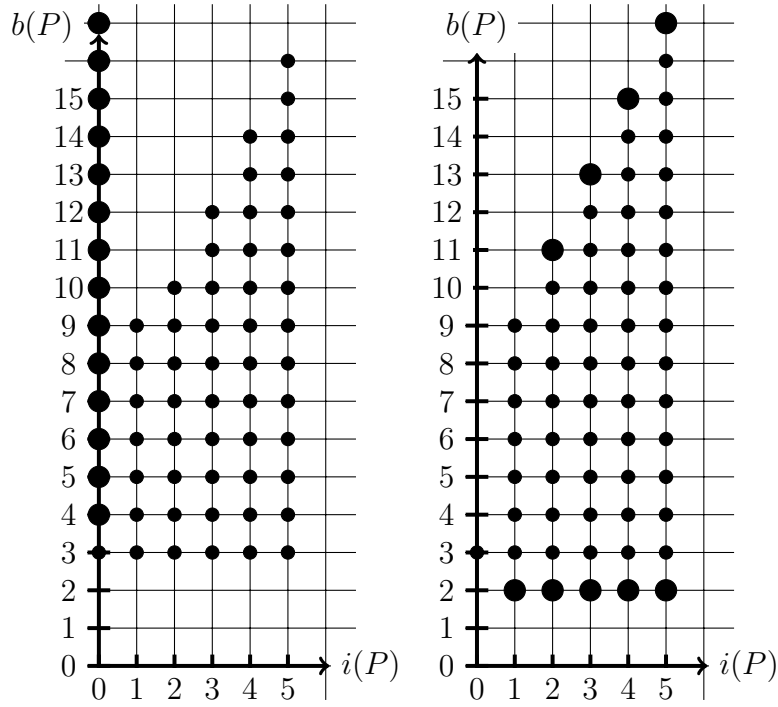
\begin{figure}[H]
	\begin{tikzpicture}[x=0.5cm,y=0.5cm]
		\draw[step=1.0,black,thin,xshift=0cm,yshift=0cm] (-0.9,17.5) grid (6.8,-0.9);
		\draw[step=1.0,black,thin,xshift=0cm,yshift=0cm] (9.1,17.5) grid (16.8,-0.9);
		
		\draw [line width=1.5pt,->] (0,0) -- (6.1,0);
		\draw [line width=1.5pt,->] (0,0) -- (0,16.7);
		
		\foreach \n in {0,...,5}{\draw [line width=1.5pt] (\n,-0.3) -- (\n,0.3);}
		\foreach \n in {0,...,5}{\draw[color=black] (\n,-0.6) node[fill=white] {$\n$ };}
		\foreach \n in {0,...,17}{\draw [line width=1.5pt] (-0.3,\n) -- (0.3,\n);}
		\foreach \n in {0,...,15}{\draw[color=black] (-0.8,\n) node[fill=white] {$\n$ };}
		
		\foreach \n in {3,...,17}{\draw [fill=black] (0,\n) circle (2.5pt);}
		\foreach \n in {4,...,17}{\draw [fill=black] (0,\n) circle (4pt);}
		\foreach \n in {3,...,9}{\draw [fill=black] (1,\n) circle (2.5pt);}
		\foreach \n in {3,...,10}{\draw [fill=black] (2,\n) circle (2.5pt);}
		\foreach \n in {3,...,12}{\draw [fill=black] (3,\n) circle (2.5pt);}
		\foreach \n in {3,...,14}{\draw [fill=black] (4,\n) circle (2.5pt);}
		\foreach \n in {3,...,16}{\draw [fill=black] (5,\n) circle (2.5pt);}
		
		\draw[color=black] (7.2,0) node[fill=white] {$i(P)$ };
		\draw[color=black] (-1.5,16.9) node[fill=white] {$b(P)$ };

		\draw [line width=1.5pt,->] (10,0) -- (16.1,0);
		\draw [line width=1.5pt,->] (10,0) -- (10,16.2);
		
		\foreach \n in {0,...,5}{\draw [line width=1.5pt] (\n+10,-0.3) -- (\n+10,0.3);}
		\foreach \n in {0,...,5}{\draw[color=black] (\n+10,-0.6) node[fill=white] {$\n$ };}
		\foreach \n in {0,...,15}{\draw [line width=1.5pt] (9.7,\n) -- (10.3,\n);}
		\foreach \n in {0,...,15}{\draw[color=black] (9.2,\n) node[fill=white] {$\n$ };}
		
		\foreach \n in {3}{\draw [fill=black] (10,\n) circle (2.5pt);}
		\foreach \n in {1,...,5}{\draw [fill=black] (10+\n,2) circle (4pt);}
		\foreach \n in {2,...,5}{\draw [fill=black] (10+\n,2*\n+7) circle (4pt);}
		\foreach \n in {3,...,9}{\draw [fill=black] (11,\n) circle (2.5pt);}
		\foreach \n in {3,...,10}{\draw [fill=black] (12,\n) circle (2.5pt);}
		\foreach \n in {3,...,12}{\draw [fill=black] (13,\n) circle (2.5pt);}
		\foreach \n in {3,...,14}{\draw [fill=black] (14,\n) circle (2.5pt);}
		\foreach \n in {3,...,16}{\draw [fill=black] (15,\n) circle (2.5pt);}

		\draw[color=black] (17.2,0) node[fill=white] {$i(P)$ };
		\draw[color=black] (10,16.9) node[fill=white] {$b(P)$ };
		
	\end{tikzpicture}
	\caption{\label{possible_pairs}All possible pairs $(i(P),b(P))$ with $i(P)\leq 5, b(P)\leq 17$ for lattice polygons on the left and for half-integral pseudo-integral polygons, which are not lattice polygons, on the right. Points with differences are marked with big nodes.}
\end{figure}

In this paper, mainly experimental methods are used. Since we can generate many examples of pseudo-integral polygons from data of suitable lattice polygons in \cite{KKN10}, \cite{Bae25} and \cite{BS24b}, we find suitable families of pseudo-integral polygons for our theorems just by looking at these examples. Moreover, the knowledge of area bounds and suitable coordinates for half-integral polygons from \cite{Boh23} proves to be very helpful for the proof of our main theorem.

The paper is organized as follows. In section 2, we describe the Ehrhart quasi-polynomial of a half-integral polygon and characterize the period-collapse for these polygons. We also show how to classify pseudo-integral polygons of denominator $2$ using huge lists of lattice polygons. In section $3$, we give a family of pseudo-integral polygons of denominator $2$ with exactly $2$ boundary lattice points, and we will see that there is only one pseudo-integral polygon of denominator $2$ without interior lattice points. In section $4$, we classify the pseudo-integral polygons of denominator $2$ with exactly one interior lattice point, which are duals of the $30$ LDP polygons of Gorenstein index 2, as we see in section $5$. We also get a formula for the sum of the numbers of boundary lattice points in the pseudo-integral polygon of denominator $2$ and its dual polygon there. In section $6$, we obtain a two-parameter family of pseudo-integral polygons of denominator $2$ with $i\geq 1$ interior and $3\leq b\leq 2i+7$ boundary lattice points, and for the extremal polygons with $2i+7$ boundary lattice points we give a complete classification in the last section.

\section{Ehrhart quasi-polynomials and quasi-period collapse for half-integral polygons}

Since the quasi-period divides the denominator of a rational polygon by \ref{Ehrhart_theorem}, the quasi-period of  a half-integral polygon $P\subseteq \R^2$ is either $1$ or $2$. So we can describe the Ehrhart quasi-polynomial of a half-integral polygon by two polynomials $\ehr{P,even}{t}$ and $\ehr{P,odd}{t}$ of degree $2$, which are defined by the values on integers $k\in \Z$ as
\begin{align*}
\ehr{P}{k}=\begin{cases}
\ehr{P,even}{k} & \text{ for } k \equiv 0 \mod 2\\
\ehr{P,odd}{k} & \text{ for } k \equiv 1 \mod 2.
\end{cases}
\end{align*}

The even part is easy to understand via the lattice polygon $2P$ and its Ehrhart polynomial $\ehr{2P}{t}$, and the linear coefficient of $\ehr{P,odd}{t}$ is $\frac{b(P)}{2}$ as in the case of a lattice polygon, which was already seen in \cite[Lemma 3.3]{Her10}. We give the complete description and argumentation for all coefficients in the following proposition.

\begin{prop}
Let $P\subseteq \R^2$ be a half-integral polygon. Then $\ehr{P}{t}$ is given by
\begin{align*}
\ehr{P,odd}{t}=&\area{P}t^2+\frac{b(P)}{2}t+l(P)-\area{P}-\frac{b(P)}{2}\\
\ehr{P,even}{t}=&\area{P}t^2+\frac{b(2P)}{4}t+1.
\end{align*}
\end{prop}
\begin{proof}
From \ref{Ehrhart_theorem} we know, that there exist rational numbers $e_0$ and $e_1$ with
\begin{align*}
\ehr{P,odd}{t}=&\area{P}t^2+e_1t+e_0.
\end{align*}
Using the Ehrhart-Macdonald reciprocity \ref{Ehrhart-Macdonald}, we have $\ehr{P}{-1}=i(P)$ and so we get
\begin{align*}
e_1=\frac{\ehr{P,odd}{1}-\ehr{P,odd}{-1}}{2}=\frac{\ehr{P}{1}-\ehr{P}{-1}}{2}=\frac{l(P)-i(P)}{2}=\frac{b(P)}{2}.
\end{align*}
Now it follows from $\ehr{P}{1}=\ehr{P,odd}{1}=l(P)$ that
\begin{align*}
e_0=\ehr{P,odd}{1}-\area{P}-e_1=l(P)-\area{P}-\frac{b(P)}{2}.
\end{align*}
We have for all $k\in \Z_{\geq 0}$ that $\ehr{P,even}{2k}=l(2kP)=\ehr{2P}{k}$ and so we get $\ehr{P,even}{2t}=\ehr{2P}{t}$. We know the Ehrhart polynomial of the lattice polygon $2P$ from the introduction, and with the quadratic scaling of the area we have
\begin{align*}
\ehr{P,even}{t}=\ehr{2P}{\frac{t}{2}}=\area{2P}\frac{t^2}{4}+\frac{b(2P)}{2}\frac{t}{2}+1=\area{P}t^2+\frac{b(2P)}{4}t+1.
\end{align*}
\end{proof}

Knowing the Ehrhart quasi-polynomial of a half-integral polygon, we can now directly characterize the period collapse of its coefficient functions.

\begin{coro}\label{pseudo_crit}
Let $P\subseteq \R^2$ be a half-integral polygon. Then we have the period sequence $(1,\ast,1)$ for $P$ if and only if $l(P)-\area{P}-\frac{b(P)}{2}=1$, i.e., $P$ obeys Pick's formula, and $(\ast,1,1)$ if and only if for every edge $e$ of $P$ we have $e \cap \Z^2 \neq \emptyset$. In particular, $P$ is pseudo-integral if and only if it obeys Pick's formula and has a lattice point on every edge.
\end{coro}
\begin{proof}
The only thing to show is that $\frac{b(P)}{2}=\frac{b(2P)}{4}$ if and only if every edge of $P$ contains a lattice point. But we have a lattice point on every edge if and only if there is always exactly one half-integral boundary point between two neighboring boundary lattice points, and this is equivalent to $2b(P)=b(2P)$, since $b(2P)$ counts the half-integral boundary points of $P$.
\end{proof}

Pseudo-integral polygons are also generally characterized using Pick's formula and boundary lattice points of dilates, as the following theorem shows.

\begin{thm}[\protect{\cite[3.1]{MAW05}}]\label{pseudo_crit_general}
Let $P\subseteq \R^2$ be a rational polygon. Then the following conditions are equivalent.
\begin{itemize}
	\item $P$ is pseudo-integral.
	\item $\ehr{P}{t}=\area{P}t^2+\frac{b(P)}{2}t+1$.
	\item $kP$ obeys Pick's theorem and $b(kP)=kb(P)$ for all $k\in \Z_{\geq 1}, k\leq \denom{P}$.
\end{itemize}
\end{thm}

Therefore, we can compute the number of interior and boundary lattice points of $\denom{P}\cdot P$ for any pseudo-integral polygon using the following corollary.

\begin{coro}\label{int_bound_points}
Let $P\subseteq \R^2$ be a pseudo-integral polygon with $d:=\denom{P}$. Then $dP$ is a lattice polygon with $i(dP)=d^2\cdot i(P)+\frac{d^2-d}{2}\cdot b(P)-d^2+1$ and $b(dP)=d\cdot b(P)$.
\end{coro}
\begin{proof}
From theorem \ref{pseudo_crit_general} we have $b(dP)=d\cdot b(P)$ and can calculate $i(dP)$ with the Ehrhart-Macdonald reciprocity \ref{Ehrhart-Macdonald} as
\begin{align*}
i(dP)=&\ehr{P}{-d}=\area{P}\cdot d^2-\frac{b(P)}{2}\cdot d+1\\=&d^2\left(i(P)+\frac{b(P)}{2}-1\right)-\frac{b(P)}{2}\cdot d+1\\
=&d^2\cdot i(P)+\frac{d^2-d}{2}\cdot b(P)-d^2+1.
\end{align*}
\end{proof}

Specializing to $\denom{P}=2$ we get the following way to classify half-integral pseudo-integral polygons.

\begin{coro}\label{detecting_half_integral_pseudo_integral}
Let $P\subseteq \R^2$ be a half-integral pseudo-integral polygon with $n$ vertices. Then $P$ is affine unimodular equivalent to a polygon with vertices
\begin{align*}
&\left\{\frac{1}{2}v_i\right\}_{1\leq i\leq n}, \left\{\frac{1}{2}v_i-\left(0,\frac{1}{2}\right)\right\}_{1\leq i\leq n}, \left\{\frac{1}{2}v_i-\left(\frac{1}{2},0\right)\right\}_{1\leq i\leq n} \mbox{ or }\\ &\left\{\frac{1}{2}v_i-\left(\frac{1}{2},\frac{1}{2}\right)\right\}_{1\leq i\leq n},
\end{align*} where $\{v_i\}_{1\leq i\leq n}$ are the vertices of a lattice polygon $Q$ with $i(Q)=4i(P)+b(P)-3$ and $b(Q)=2b(P)$.
\end{coro}

There are several exhaustive classifications of lattice polygons, according to the number of lattice points (up to $l(P)=42$ in \cite{Koe91}), the number of interior points (up to $i(P)=30$ in \cite{Cas12}), and the area (up to $\area{a(P)}=25$ in \cite{Bal21}). Koelman's list was recently extended in \cite{BS24b}.

From all these lists we can create sublists of lattice polygons with a fixed number of interior and boundary lattice points. To classify up to affine unimodular equivalence all half-integral pseudo-integral polygons $P\subseteq \R^2$ with a given number of interior and boundary lattice points, we can test Pick's formula and $2b(P)=b(2P)$ for all polygons from the sublists that are suitable according to \ref{detecting_half_integral_pseudo_integral}. We did this using lattice polygons with up to $78$ lattice points and got the following result (for the complete list of the polygons in the result see \cite{Boh24a}).

\begin{classify}\label{classification_small_number_interior_points}
Up to affine unimodular equivalence there are exactly $16688$ pseudo-integral polygons $P$ with denominator $2$ and at most $6$ interior lattice points. Depending on $i(P)$ and $b(P)$ the polygons are distributed as follows. We give not only the number of pseudo-integral polygons with denominator $2$ at $\# (d=2)$, but also the number of lattice polygons at $\# (d=1)$ for comparison.	
\begin{center}
\begin{footnotesize}
\begin{tabular}{c||c||c|c|c|c|c|c|c|c}
$(i,b)$&$(0,3)$&$(1,2)$&$(1,3)$&$(1,4)$&$(1,5)$&$(1,6)$&$(1,7)$&$(1,8)$&$(1,9)$\\\hline
$\# (d=1)$&$1$&$0$&$1$&$3$&$2$&$4$&$2$&$3$&$1$\\\hline
$\# (d=2)$&$1$&$6$&$6$&$4$&$7$&$3$&$2$&$1$&$1$
\end{tabular}
		
\vspace{0.5cm}
		
\begin{tabular}{c||c|c|c|c|c|c|c|c|c|c}
$(i,b)$&$(2,2)$&$(2,3)$&$(2,4)$&$(2,5)$&$(2,6)$&$(2,7)$&$(2,8)$&$(2,9)$&$(2,10)$&$(2,11)$\\\hline
$\# (d=1)$&$0$&$1$&$5$&$5$&$11$&$7$&$9$&$3$&$4$&$0$\\\hline
$\# (d=2)$&$8$&$35$&$59$&$39$&$27$&$27$&$11$&$7$&$5$&$2$
\end{tabular}
		
\vspace{0.5cm}
		
\begin{tabular}{c||c|c|c|c|c|c}
$(i,b)$&$(3,2)$&$(3,3)$&$(3,4)$&$(3,5)$&$(3,6)$&$(3,7)$\\\hline
$\# (d=1)$&$0$&$2$&$8$&$12$&$19$&$17$ \\\hline
$\# (d=2)$&$29$&$103$&$138$&$124$&$122$&$72$ \\\hline\hline
$(i,b)$&$(3,8)$&$(3,9)$&$(3,10)$&$(3,11)$&$(3,12)$&$(3,13)$\\\hline
$\# (d=1)$&$23$&$14$&$14$&$5$&$6$&$0$\\\hline
$\# (d=2)$&$44$&$39$&$15$&$13$&$8$&$4$
\end{tabular}
		
\vspace{0.5cm}
		
\begin{tabular}{c||c|c|c|c|c|c|c}
$(i,b)$&$(4,2)$&$(4,3)$&$(4,4)$&$(4,5)$&$(4,6)$&$(4,7)$&$(4,8)$\\\hline
$\# (d=1)$&$0$&$1$&$10$&$15$&$33$&$29$&$31$ \\\hline
$\# (d=2)$&$29$&$224$&$400$&$366$&$270$&$164$&$148$ \\\hline\hline
$(i,b)$&$(4,9)$&$(4,10)$&$(4,11)$&$(4,12)$&$(4,13)$&$(4,14)$&$(4,15)$\\\hline
$\# (d=1)$&$22$&$27$&$15$&$17$&$5$&$6$&$0$\\\hline
$\# (d=2)$&$86$&$48$&$42$&$14$&$10$&$7$&$3$
\end{tabular}
		
\vspace{0.5cm}
		
\begin{tabular}{c||c|c|c|c|c|c|c|c}
$(i,b)$&$(5,2)$&$(5,3)$&$(5,4)$&$(5,5)$&$(5,6)$&$(5,7)$&$(5,8)$&$(5,9)$\\\hline
$\# (d=1)$&$0$&$2$&$16$&$28$&$52$&$61$&$61$&$46$ \\\hline
$\# (d=2)$&$44$&$420$&$900$&$1035$&$784$&$482$&$271$&$160$ \\\hline\hline
$(i,b)$&$(5,10)$&$(5,11)$&$(5,12)$&$(5,13)$&$(5,14)$&$(5,15)$&$(5,16)$&$(5,17)$\\\hline
$\# (d=1)$&$36$&$25$&$28$&$17$&$18$&$6$&$7$&$0$\\\hline
$\# (d=2)$&$159$&$90$&$55$&$49$&$16$&$12$&$8$&$4$
\end{tabular}
		
\vspace{0.5cm}
		
\begin{tabular}{c||c|c|c|c|c|c|c|c|c}
$(i,b)$&$(6,2)$&$(6,3)$&$(6,4)$&$(6,5)$&$(6,6)$&$(6,7)$&$(6,8)$&$(6,9)$&$(6,10)$\\\hline
$\# (d=1)$&$0$&$3$&$17$&$39$&$84$&$92$&$111$&$87$&$76$ \\\hline
$\# (d=2)$&$80$&$718$&$1868$&$2148$&$1664$&$1111$&$663$&$367$&$217$ \\\hline\hline
$(i,b)$&$(6,11)$&$(6,12)$&$(6,13)$&$(6,14)$&$(6,15)$&$(6,16)$&$(6,17)$&$(6,18)$&$(6,19)$\\\hline
$\# (d=1)$&$49$&$40$&$26$&$34$&$20$&$21$&$7$&$8$&$0$\\\hline
$\# (d=2)$&$150$&$174$&$103$&$63$&$56$&$18$&$13$&$9$&$4$
\end{tabular}
\end{footnotesize}		
\end{center}
\end{classify}
\begin{proof}
The classification algorithm has already been justified above. It only remains to show that every pseudo-integral polygon $P$ with denominator $2$ has $b(P)=3$ if $i(P)=0$ and it has $b(P)\leq 2i(P)+7$ if $i(P)>0$. This is done in \ref{pseudo-integral_without_interior_lattice points} and in \ref{upper_bound_boundary_points}.
\end{proof}

We see that there are examples with $b(P)=2i(P)+7$ and $i(P)>0$, which are not possible for lattice polygons by Scott's inequality \ref{Scott_inequality}. There are also examples with $b(P)=2$, which we will discuss in the next section. We end this section by defining a special lattice polytope for each rational polytope $P$, which will be very helpful in the following.

\begin{defi}
Let $P\subseteq \R^n$ be a rational polytope.\\
Then the \textit{integer hull of $P$} is the lattice polytope $\conv{P\cap \Z^n}$, and we denote it by $\inthull{P}$.
\end{defi}

\begin{prop}\label{int_hull}
Let $P\subseteq \R^2$ be a pseudo-integral polygon that is not a lattice polygon. If $\dim(\inthull{P})=2$, then the following statements hold:
\begin{itemize}
	\item We have $i(P)>i(\inthull{P})$.
	\item We have $b(P)=b(\inthull{P})-(i(P)-i(\inthull{P}))$.
	\item We have $\area{P}=\area{\inthull{P}}+\frac{1}{2}(i(P)-i(\inthull{P}))$.
\end{itemize}
\end{prop}
\begin{proof}
Since $\denom{P}>1$, we have $\inthull{P}\subsetneq P$, and therefore we also get $\area{\inthull{P}}<\area{P}$. By definition, $\inthull{P}$ and $P$ have the same number of lattice points. Since $P$ as a pseudo-integral polygon obeys Pick's formula by \ref{pseudo_crit_general},  we obtain the claimed identities.
\end{proof}

\section{Half-integral pseudo-integral polygons with at most two boundary lattice points or without interior lattice points}

There are no pseudo-integral polygons without boundary lattice points since by \ref{pseudo_crit_general} a pseudo-integral polygon with denominator $d$ has at least
$b(P)=\frac{b(dP)}{d}\geq \ceil{\frac{3}{d}}\geq 1$ boundary lattice point, since the lattice polygon $dP$ has at least $3$ boundary lattice points. This was already seen in \cite[Theorem 1.2]{MAM17}. However, for any given number of interior lattice points $i>1$ there are examples of pseudo-integral polygons with one boundary lattice point and denominator $2i+1$ in \cite[Theorem 1.2]{MAM17}. We can also construct a family of pseudo-integral triangles of denominator $3$ with exactly one boundary and $i$ interior lattice points by
\begin{align*}
T_{3,(i,1)}=\conv{\left(\frac{2}{3},0\right), (1,0), \left(\frac{1}{3}+i,6i-3\right)}.
\end{align*}

In the half-integral case, this cannot happen because we have a lattice point on every edge by \ref{pseudo_crit}.

Nevertheless, we have half-integral pseudo-integral polygons with two boundary lattice points for any given number of interior lattice points greater than $0$. For example, we can get such polygons from the family of polygons in the following proposition. 

\begin{prop}
Let be $i\in \Z_{\geq 1}$. Then
\begin{align*}
T_{2,(i,2)}:=\conv{(0,1), (1,-1), \left(i+\frac{1}{2},0\right)}\subseteq \R^2
\end{align*}
is a half-integral pseudo-integral triangle with $i(T_{2,(i,2)})=i$ and $b(T_{2,(i,2)})=2$.
\end{prop}
\begin{proof}
$T_{2,(i,2)}$ is half-integral, $|\mathrm{int}(T_{2,(i,2)})\cap \Z^2|=|\{(k,0) \mid k\in \Z, 1\leq k\leq i\}|=i$, and
\begin{align*}
b(T_{2,(i,2)})=|\{(0,1), (1,-1)\}|=2.
\end{align*}
Since $\area{T_{2,(i,2)}}=i$ we get $\area{T_{2,(i,2)}}=i(T_{2,(i,2)})+\frac{b(T_{2,(i,2)})}{2}-1$ and we have lattice points on every edge of $T_{2,(i,2)}$, so with \ref{pseudo_crit} we also get that $T_{2,(i,2)}$ is pseudo-integral. 
\end{proof}

\begin{figure}[H]
	\begin{tikzpicture}[x=0.6cm,y=0.6cm]
		\draw[step=2.0,black,thin,xshift=0cm,yshift=0cm] (-5,3) grid (17.9,-3);
		
		\draw [line width=1pt,color=black] (-4,2)-- (-2,-2);
		\draw [line width=1pt,color=black] (-2,-2)-- (-1,0);
		\draw [line width=1pt,color=black] (-1,0)-- (-4,2);
		
		\draw [fill=black] (-4,2) circle (2.5pt);
		\draw [fill=black] (-2,-2) circle (2.5pt);
		\draw [fill=black] (-1,0) circle (2.5pt);
		\draw [fill=black] (-2,0) circle (2.5pt);
		
		\draw[color=black] (-1,1) node {$T_{2,(1,2)}$ };

		\draw [line width=1pt,color=black] (2,2)-- (4,-2);
		\draw [line width=1pt,color=black] (4,-2)-- (7,0);
		\draw [line width=1pt,color=black] (7,0)-- (2,2);
		
		\draw [fill=black] (2,2) circle (2.5pt);
		\draw [fill=black] (4,-2) circle (2.5pt);
		\draw [fill=black] (7,0) circle (2.5pt);
		\draw [fill=black] (4,0) circle (2.5pt);
		\draw [fill=black] (6,0) circle (2.5pt);
		
		\draw[color=black] (7,1) node {$T_{2,(2,2)}$ };

		\draw [line width=1pt,color=black] (10,2)-- (12,-2);
		\draw [line width=1pt,color=black] (12,-2)-- (17,0);
		\draw [line width=1pt,color=black] (17,0)-- (10,2);
		
		\draw [fill=black] (10,2) circle (2.5pt);
		\draw [fill=black] (12,-2) circle (2.5pt);
		\draw [fill=black] (17,0) circle (2.5pt);
		\draw [fill=black] (12,0) circle (2.5pt);
		\draw [fill=black] (14,0) circle (2.5pt);
		\draw [fill=black] (16,0) circle (2.5pt);
		
		\draw[color=black] (17,1) node {$T_{2,(3,2)}$ };
		
	\end{tikzpicture}
	\caption{Triangles, which are affine unimodular equivalent to the triangles $T_{2,(1,2)}$, $T_{2,(2,2)}$ and $T_{2,(3,2)}$. }
\end{figure}
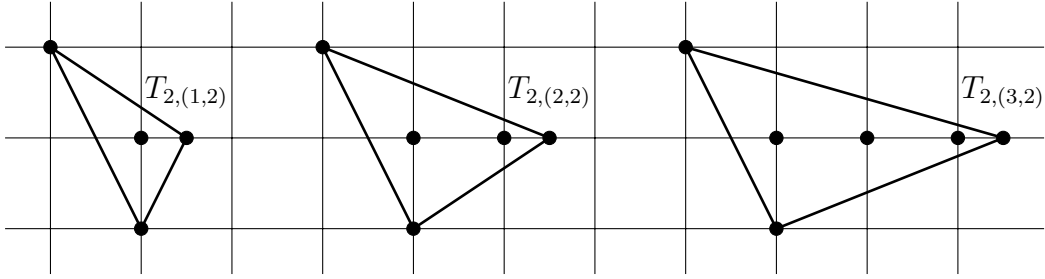

We now turn to polygons without interior lattice points, which is a rich class for lattice polygons, but consists of only one polygon for half-integral pseudo-integral polygons, as we will see.

There is no pseudo-integral polygon $P$ without interior lattice points with less than $3$ boundary lattice points since $P$ must obey Pick's formula by \ref{pseudo_crit} and so we have
\begin{align*}
b(P)=2\cdot \area{P}-2i(P)+2=2\cdot \area{P}+2>2.
\end{align*}
This has already been seen in \cite[1.2]{MAM17}. There are also no non-integral pseudo-integral polygons without interior lattice points and non-collinear boundary lattice points, as the following lemma shows.

\begin{lemma}\label{hollow_non_collinear_boundary_points}
Let $P\subseteq \R^2$ be pseudo-integral polygon without interior lattice points and with non-collinear boundary lattice points. Then $P$ is a lattice polygon.
\end{lemma}
\begin{proof}
Since the boundary lattice points are not collinear, $\dim(\inthull{P})=2$. So the lemma follows from \ref{int_hull}.
\end{proof}

It remains to examine polygons with collinear boundary lattice points. For half-integral polygons we get only one such polygon, which is a non-lattice pseudo-integral polygon.

\begin{prop}\label{pseudo-integral_without_interior_lattice points}
Let $P\subseteq \R^2$ be half-integral pseudo-integral polygon with $i(P)=0$. \\
Then $P$ is either a lattice polygon or $P$ is affine unimodular equivalent to the triangle 
\begin{align*}
T_{2,(0,3)}:=\conv{(0,0),(2,0),\left(0,\frac{1}{2}\right)}
\end{align*}
\end{prop}
\begin{proof}
If the boundary lattice points are not collinear, we have a lattice polygon by \ref{hollow_non_collinear_boundary_points}. 

If the boundary lattice points are collinear, $P$ is a triangle and $e:=\inthull{P}$ is an edge of $P$, since we must have a lattice point on every edge of $P$ by \ref{pseudo_crit}. If $d$ is the lattice distance of the non-integral vertex from the edge $e$ and $\mathrm{length}(e):=l(e)-1$ the lattice length of $e$, then we have $2d\in \Z_{\geq 1}$ and with Pick's formula
\begin{align*}
\frac{b(P)}{2}-1=\area{P}=\mathrm{length}(e)\cdot \frac{d}{2}=(b(P)-1)\cdot \frac{d}{2}.
\end{align*}
Therefore $2d=\frac{2b(P)-4}{b(P)-1}<2$. We get $d=\frac{1}{2}$ and $b(P)=3$, so $P$ is affine unimodular equivalent to $\conv{(0,0),(2,0),(0,\frac{1}{2})}$.
\end{proof}

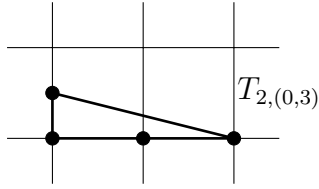
\begin{figure}[H]
	\begin{tikzpicture}[x=0.6cm,y=0.6cm]
		\draw[step=2.0,black,thin,xshift=0cm,yshift=0cm] (-1,3) grid (5,-1);
		
		\draw [line width=1pt,color=black] (0,0)-- (4,0);
		\draw [line width=1pt,color=black] (4,0)-- (0,1);
		\draw [line width=1pt,color=black] (0,1)-- (0,0);
		
		\draw [fill=black] (0,0) circle (2.5pt);
		\draw [fill=black] (4,0) circle (2.5pt);
		\draw [fill=black] (2,0) circle (2.5pt);
		\draw [fill=black] (0,1) circle (2.5pt);
		
		\draw[color=black] (5,1) node {$T_{2,(0,3)}$ };
		
	\end{tikzpicture}
	\caption{Up to affine unimodular equivalence the only pseudo-integral polygon with $\denom{P}=2$ and $i(P)=0$. }
\end{figure}

\begin{rem}\label{hollow_lattice_polygons}
The lattice polygons without interior lattice points are classified (e.g. \cite[4.1.2]{Koe91}). They are either affine unimodular equivalent to the lattice triangle $\conv{(0,0),(2,0),(0,2)}$ or to exactly one of the lattice polygons
\begin{align*}
P_{1,(0,(u,d))}=\conv{(1,0),(d,0),(u,1),(1,1)} \text{ with } u,d\in \Z_{\geq 1}, d>1, d\geq u.
\end{align*}
In particular, since $b(P_{1,(0,(u,d))})=u+d$, there is a lattice polygon without interior lattice points for any given number of boundary lattice points $b\geq 3$.
\end{rem}

\begin{rem}
The situation is quite different for denominators greater than $2$. We cannot generally expect that there is only a finite number of pseudo-integral polygons without interior lattice points for a fixed denominator $d>2$. For example, for $d=3$ and any  $b\in \Z_{\geq 4}$, one can show that there is a pseudo-integral polygon
\begin{align*}
Q_{3,(0,b)}:=\conv{(0,0),\left(\frac{1}{3},0\right),\left(b-\frac{7}{3},2b-5\right), (b-1,2b-2)}
\end{align*} 
without interior lattice points and with
\begin{align*}
b(Q_{3,(0,b)})=|\{(l,2l) \mid l\in\Z, 0\leq l\leq b-1\}|=b.
\end{align*}
In particular, for every $b\in \Z_{\geq 3}$ we have a pseudo-integral polygon with denominator $2$ or $3$, without interior lattice points and with $b$ boundary lattice points.
\end{rem}

\begin{rem}
The triangle in \ref{pseudo-integral_without_interior_lattice points} is affine unimodular equivalent to the triangle $\conv{(0,0),(\frac{1}{2},0),(2,2)}$. It seems to be possible to generalize this example and get for each denominator $d\in\Z_{\geq 2}$ a triangle
\begin{align*}
T_{d,(0,d+1)}:=\conv{(0,0),\left(\frac{1}{d},0\right),(d,(d-1)d)},
\end{align*}
which is pseudo-integral without interior lattice points and with the $d+1$ boundary lattice points of the set $\{(l,(d-1)l) \mid l\in \Z, 0\leq l\leq d\}$.
\end{rem}

\section{Half-integral pseudo-integral polygons with exactly one interior lattice point}

For lattice polygons, up to affine unimodular equivalence, we have exactly $16$ polygons $P$ with $i(P)=1$, the $16$ \textit{reflexive polygons}, independently classified in \cite{Bat85}, \cite{Rab89} and \cite{Koe91}.  Pseudo-integral polygons with denominator $2$ and exactly one interior lattice point have some nice properties in common with reflexive polygons, which we will use for a deeper investigation in the next section. In this section, we give the complete classification of pseudo-integral polygons with denominator $2$ and exactly one interior lattice point without using these properties.

\begin{classify}
There are exactly $30$ pseudo-integral polygons $P\subseteq \R^2$ with denominator $2$ and $i(P)=1$ up to affine unimodular equivalence. Each of them is affine unimodular equivalent to exactly one of the polygons in figure \ref{1_interior_lattice_point}.
\end{classify}
\begin{proof}
Since $\denom{P}=2$, the polygon $P$ has at least $2$ lattice points due to \ref{pseudo_crit}, and so $\inthull{P}$ is at least of dimension $1$.
	
We first consider the case $\dim(\inthull{P})=1$. Since $\denom{P}=2$, by \ref{pseudo_crit} we have a lattice point on every edge of $P$, and so with $i(P)=1$ we get that $P$ is the convex hull of $\inthull{P}$, which must be a segment with $3$ lattice points, and of two half-integral points, one on each side of the affine hull of $\inthull{P}$. So we have $b(P)=2$ and with Pick's formula we get $\area{P}=1$, so that the two half-integral vertices have lattice distance $\frac{1}{2}$ to $\inthull{P}$. After a suitable affine unimodular transformation we can assume that $\inthull{P}=\conv{(0,1),(0,-1)}$ and the half-integral point $(\frac{1}{2},-1)$ is a vertex of $P$. Because of the convexity of $P$, the other half-integral vertex of $P$ must be a point of $\{(-\frac{1}{2},\frac{l}{2}-1) \mid l\in\Z, 0\leq l \leq 8\}$.  With suitable affine unimodular transformations, we can limit ourselves to $l\leq 4$, so we get up to affine unimodular equivalence one of the first five polygons in the first row of figure \ref{1_interior_lattice_point}.
	
Now we turn to the case where $\inthull{P}$ is of dimension $2$. Then we have by \ref{int_hull} $i(\inthull{P})=0$ and $b(\inthull{P})=b(P)+1$, since $i(P)=1$. Moreover, by \ref{int_hull} we get the relation $\area{P}=\area{\inthull{P}}+\frac{1}{2}$. Therefore, we obtain $P$ as the convex hull of $\inthull{P}$ and a single half-integral point not in $\inthull{P}$, and the closure in $\R^2$ of $P\setminus \inthull{P}$ is either a triangle affine unimodular equivalent to the triangle $\conv{(0,0), (2,0), (0,\frac{1}{2})}$, or the union of two triangles sharing an edge, while each of them is affine unimodular equivalent to $\conv{(0,0), (1,0), (0,\frac{1}{2})}$. In particular, $\inthull{P}$ has one edge with exactly $3$ lattice points or two adjacent edges with exactly $2$ lattice points.

\begin{figure}[H]
	\begin{tikzpicture}[x=0.42cm,y=0.42cm]
		\draw[step=2.0,black,thin,xshift=0cm,yshift=0cm] (-1.9,3.7) grid (23,-45.5);

		\draw[color=black] (0,2.8) node[fill=white] {$b(P)=2$ };
		
		\draw [line width=1pt,color=black] (-1,2)-- (0,-2);
		\draw [line width=1pt,color=black] (0,-2)-- (1,-2);
		\draw [line width=1pt,color=black] (1,-2)-- (0,2);
		\draw [line width=1pt,color=black] (0,2)-- (-1,2);
		
		\draw [fill=black] (0,0) circle (2.5pt);
		\draw [fill=black] (0,2) circle (2.5pt);
		\draw [fill=black] (0,-2) circle (2.5pt);
		\draw [fill=black] (1,-2) circle (2.5pt);
		\draw [fill=black] (-1,2) circle (2.5pt);

		\draw [line width=1pt,color=black] (3,1)-- (4,-2);
		\draw [line width=1pt,color=black] (4,-2)-- (5,-2);
		\draw [line width=1pt,color=black] (5,-2)-- (4,2);
		\draw [line width=1pt,color=black] (4,2)-- (3,1);
		
		\draw [fill=black] (4,0) circle (2.5pt);
		\draw [fill=black] (4,2) circle (2.5pt);
		\draw [fill=black] (4,-2) circle (2.5pt);
		\draw [fill=black] (5,-2) circle (2.5pt);
		\draw [fill=black] (3,1) circle (2.5pt);

		\draw [line width=1pt,color=black] (7,0)-- (8,-2);
		\draw [line width=1pt,color=black] (8,-2)-- (9,-2);
		\draw [line width=1pt,color=black] (9,-2)-- (8,2);
		\draw [line width=1pt,color=black] (8,2)-- (7,0);
		
		\draw [fill=black] (8,0) circle (2.5pt);
		\draw [fill=black] (8,2) circle (2.5pt);
		\draw [fill=black] (8,-2) circle (2.5pt);
		\draw [fill=black] (9,-2) circle (2.5pt);
		\draw [fill=black] (7,0) circle (2.5pt);

		\draw [line width=1pt,color=black] (11,-1)-- (12,-2);
		\draw [line width=1pt,color=black] (12,-2)-- (13,-2);
		\draw [line width=1pt,color=black] (13,-2)-- (12,2);
		\draw [line width=1pt,color=black] (12,2)-- (11,-1);
		
		\draw [fill=black] (12,0) circle (2.5pt);
		\draw [fill=black] (12,2) circle (2.5pt);
		\draw [fill=black] (12,-2) circle (2.5pt);
		\draw [fill=black] (13,-2) circle (2.5pt);
		\draw [fill=black] (11,-1) circle (2.5pt);

		\draw [line width=1pt,color=black] (15,-2)-- (17,-2);
		\draw [line width=1pt,color=black] (17,-2)-- (16,2);
		\draw [line width=1pt,color=black] (16,2)-- (15,-2);
		
		\draw [fill=black] (16,0) circle (2.5pt);
		\draw [fill=black] (16,2) circle (2.5pt);
		\draw [fill=black] (16,-2) circle (2.5pt);
		\draw [fill=black] (17,-2) circle (2.5pt);
		\draw [fill=black] (15,-2) circle (2.5pt);

		\draw [line width=1pt,color=black] (19,1)-- (20,-2);
		\draw [line width=1pt,color=black] (20,-2)-- (22,0);
		\draw [line width=1pt,color=black] (22,0)-- (19,1);		
		\draw [fill=black] (20,-2) circle (2.5pt);
		\draw [fill=black] (22,0) circle (2.5pt);
		\draw [fill=black] (20,0) circle (2.5pt);
		\draw [fill=black] (19,1) circle (2.5pt);

		\draw[color=black] (0,-3.2) node[fill=white] {$b(P)=3$ };
		
		\draw [line width=1pt,color=black] (-1,-4)-- (0,-8);
		\draw [line width=1pt,color=black] (0,-8)-- (2,-4);
		\draw [line width=1pt,color=black] (2,-4)-- (0,-4);
		\draw [line width=1pt,color=black] (0,-4)-- (-1,-4);
		
		\draw [fill=black] (0,-4) circle (2.5pt);
		\draw [fill=black] (0,-6) circle (2.5pt);
		\draw [fill=black] (0,-8) circle (2.5pt);
		\draw [fill=black] (2,-4) circle (2.5pt);
		\draw [fill=black] (-1,-4) circle (2.5pt);

		\draw [line width=1pt,color=black] (3,-5)-- (4,-8);
		\draw [line width=1pt,color=black] (4,-8)-- (6,-4);
		\draw [line width=1pt,color=black] (6,-4)-- (4,-4);
		\draw [line width=1pt,color=black] (4,-4)-- (3,-5);
		
		\draw [fill=black] (4,-4) circle (2.5pt);
		\draw [fill=black] (4,-6) circle (2.5pt);
		\draw [fill=black] (4,-8) circle (2.5pt);
		\draw [fill=black] (6,-4) circle (2.5pt);
		\draw [fill=black] (3,-5) circle (2.5pt);

		\draw [line width=1pt,color=black] (7,-6)-- (8,-8);
		\draw [line width=1pt,color=black] (8,-8)-- (10,-4);
		\draw [line width=1pt,color=black] (10,-4)-- (8,-4);
		\draw [line width=1pt,color=black] (8,-4)-- (7,-6);
		
		\draw [fill=black] (8,-4) circle (2.5pt);
		\draw [fill=black] (8,-6) circle (2.5pt);
		\draw [fill=black] (8,-8) circle (2.5pt);
		\draw [fill=black] (10,-4) circle (2.5pt);
		\draw [fill=black] (7,-6) circle (2.5pt);

		\draw [line width=1pt,color=black] (11,-7)-- (12,-8);
		\draw [line width=1pt,color=black] (12,-8)-- (14,-4);
		\draw [line width=1pt,color=black] (14,-4)-- (12,-4);
		\draw [line width=1pt,color=black] (12,-4)-- (11,-7);
		
		\draw [fill=black] (12,-4) circle (2.5pt);
		\draw [fill=black] (12,-6) circle (2.5pt);
		\draw [fill=black] (12,-8) circle (2.5pt);
		\draw [fill=black] (14,-4) circle (2.5pt);
		\draw [fill=black] (11,-7) circle (2.5pt);

		\draw [line width=1pt,color=black] (15,-5)-- (18,-10);
		\draw [line width=1pt,color=black] (18,-10)-- (18,-6);
		\draw [line width=1pt,color=black] (18,-6)-- (15,-5);
		
		\draw [fill=black] (16,-6) circle (2.5pt);
		\draw [fill=black] (18,-6) circle (2.5pt);
		\draw [fill=black] (18,-8) circle (2.5pt);
		\draw [fill=black] (18,-10) circle (2.5pt);
		\draw [fill=black] (15,-5) circle (2.5pt);

		\draw [line width=1pt,color=black] (19,-5)-- (20,-8);
		\draw [line width=1pt,color=black] (20,-8)-- (22,-8);
		\draw [line width=1pt,color=black] (22,-8)-- (22,-6);
		\draw [line width=1pt,color=black] (22,-6)-- (19,-5);
		
		\draw [fill=black] (20,-6) circle (2.5pt);
		\draw [fill=black] (20,-8) circle (2.5pt);
		\draw [fill=black] (22,-6) circle (2.5pt);
		\draw [fill=black] (22,-8) circle (2.5pt);
		\draw [fill=black] (19,-5) circle (2.5pt);

		\draw[color=black] (0,-9.2) node[fill=white] {$b(P)=4$ };
		
		\draw [line width=1pt,color=black] (0,-10)-- (0,-13);
		\draw [line width=1pt,color=black] (0,-13)-- (4,-12);
		\draw [line width=1pt,color=black] (4,-12)-- (2,-10);
		\draw [line width=1pt,color=black] (2,-10)-- (0,-10);
		
		\draw [fill=black] (0,-10) circle (2.5pt);
		\draw [fill=black] (0,-12) circle (2.5pt);
		\draw [fill=black] (2,-10) circle (2.5pt);
		\draw [fill=black] (2,-12) circle (2.5pt);
		\draw [fill=black] (4,-12) circle (2.5pt);
		\draw [fill=black] (0,-13) circle (2.5pt);

		\draw [line width=1pt,color=black] (6,-10)-- (6,-12);
		\draw [line width=1pt,color=black] (6,-12)-- (7,-13);
		\draw [line width=1pt,color=black] (7,-13)-- (10,-12);
		\draw [line width=1pt,color=black] (10,-12)-- (8,-10);
		\draw [line width=1pt,color=black] (8,-10)-- (6,-10);
		
		\draw [fill=black] (6,-10) circle (2.5pt);
		\draw [fill=black] (6,-12) circle (2.5pt);
		\draw [fill=black] (8,-10) circle (2.5pt);
		\draw [fill=black] (8,-12) circle (2.5pt);
		\draw [fill=black] (10,-12) circle (2.5pt);
		\draw [fill=black] (7,-13) circle (2.5pt);

		\draw [line width=1pt,color=black] (12,-10)-- (12,-12);
		\draw [line width=1pt,color=black] (12,-12)-- (14,-13);
		\draw [line width=1pt,color=black] (14,-13)-- (16,-12);
		\draw [line width=1pt,color=black] (16,-12)-- (14,-10);
		\draw [line width=1pt,color=black] (14,-10)-- (12,-10);
		
		\draw [fill=black] (12,-10) circle (2.5pt);
		\draw [fill=black] (12,-12) circle (2.5pt);
		\draw [fill=black] (14,-10) circle (2.5pt);
		\draw [fill=black] (14,-12) circle (2.5pt);
		\draw [fill=black] (16,-12) circle (2.5pt);
		\draw [fill=black] (14,-13) circle (2.5pt);

		\draw [line width=1pt,color=black] (17,-11)-- (18,-14);
		\draw [line width=1pt,color=black] (18,-14)-- (22,-14);
		\draw [line width=1pt,color=black] (22,-14)-- (20,-12);
		\draw [line width=1pt,color=black] (20,-12)-- (17,-11);
		
		\draw [fill=black] (18,-12) circle (2.5pt);
		\draw [fill=black] (18,-14) circle (2.5pt);
		\draw [fill=black] (20,-12) circle (2.5pt);
		\draw [fill=black] (20,-14) circle (2.5pt);
		\draw [fill=black] (22,-14) circle (2.5pt);
		\draw [fill=black] (17,-11) circle (2.5pt);

		\draw[color=black] (0,-15.2) node[fill=white] {$b(P)=5$ };
		
		\draw [line width=1pt,color=black] (0,-16)-- (0,-19);
		\draw [line width=1pt,color=black] (0,-19)-- (4,-18);
		\draw [line width=1pt,color=black] (4,-18)-- (4,-16);
		\draw [line width=1pt,color=black] (4,-16)-- (0,-16);
		
		\draw [fill=black] (0,-16) circle (2.5pt);
		\draw [fill=black] (0,-18) circle (2.5pt);
		\draw [fill=black] (2,-16) circle (2.5pt);
		\draw [fill=black] (2,-18) circle (2.5pt);
		\draw [fill=black] (4,-16) circle (2.5pt);
		\draw [fill=black] (4,-18) circle (2.5pt);
		\draw [fill=black] (0,-19) circle (2.5pt);

		\draw [line width=1pt,color=black] (6,-16)-- (6,-18);
		\draw [line width=1pt,color=black] (6,-18)-- (7,-19);
		\draw [line width=1pt,color=black] (7,-19)-- (10,-18);
		\draw [line width=1pt,color=black] (10,-18)-- (10,-16);
		\draw [line width=1pt,color=black] (10,-16)-- (6,-16);
		
		\draw [fill=black] (6,-16) circle (2.5pt);
		\draw [fill=black] (6,-18) circle (2.5pt);
		\draw [fill=black] (8,-16) circle (2.5pt);
		\draw [fill=black] (8,-18) circle (2.5pt);
		\draw [fill=black] (10,-16) circle (2.5pt);
		\draw [fill=black] (10,-18) circle (2.5pt);
		\draw [fill=black] (7,-19) circle (2.5pt);

		\draw [line width=1pt,color=black] (12,-16)-- (12,-18);
		\draw [line width=1pt,color=black] (12,-18)-- (14,-19);
		\draw [line width=1pt,color=black] (14,-19)-- (16,-18);
		\draw [line width=1pt,color=black] (16,-18)-- (16,-16);
		\draw [line width=1pt,color=black] (16,-16)-- (12,-16);
		
		\draw [fill=black] (12,-16) circle (2.5pt);
		\draw [fill=black] (12,-18) circle (2.5pt);
		\draw [fill=black] (14,-16) circle (2.5pt);
		\draw [fill=black] (14,-18) circle (2.5pt);
		\draw [fill=black] (16,-16) circle (2.5pt);
		\draw [fill=black] (16,-18) circle (2.5pt);
		\draw [fill=black] (14,-19) circle (2.5pt);

		\draw [line width=1pt,color=black] (15,-19)-- (16,-22);
		\draw [line width=1pt,color=black] (16,-22)-- (22,-22);
		\draw [line width=1pt,color=black] (22,-22)-- (18,-20);
		\draw [line width=1pt,color=black] (18,-20)-- (15,-19);
		
		\draw [fill=black] (16,-20) circle (2.5pt);
		\draw [fill=black] (18,-20) circle (2.5pt);
		\draw [fill=black] (16,-22) circle (2.5pt);
		\draw [fill=black] (18,-22) circle (2.5pt);
		\draw [fill=black] (20,-22) circle (2.5pt);
		\draw [fill=black] (22,-22) circle (2.5pt);
		\draw [fill=black] (15,-19) circle (2.5pt);

		\draw [line width=1pt,color=black] (0,-20)-- (0,-25);
		\draw [line width=1pt,color=black] (0,-25)-- (4,-24);
		\draw [line width=1pt,color=black] (4,-24)-- (0,-20);
		
		\draw [fill=black] (0,-22) circle (2.5pt);
		\draw [fill=black] (0,-24) circle (2.5pt);
		\draw [fill=black] (2,-22) circle (2.5pt);
		\draw [fill=black] (2,-24) circle (2.5pt);
		\draw [fill=black] (0,-20) circle (2.5pt);
		\draw [fill=black] (4,-24) circle (2.5pt);
		\draw [fill=black] (0,-25) circle (2.5pt);

		\draw [line width=1pt,color=black] (6,-20)-- (6,-24);
		\draw [line width=1pt,color=black] (6,-24)-- (7,-25);
		\draw [line width=1pt,color=black] (7,-25)-- (10,-24);
		\draw [line width=1pt,color=black] (10,-24)-- (6,-20);
		
		\draw [fill=black] (6,-22) circle (2.5pt);
		\draw [fill=black] (6,-24) circle (2.5pt);
		\draw [fill=black] (8,-22) circle (2.5pt);
		\draw [fill=black] (8,-24) circle (2.5pt);
		\draw [fill=black] (6,-20) circle (2.5pt);
		\draw [fill=black] (10,-24) circle (2.5pt);
		\draw [fill=black] (7,-25) circle (2.5pt);

		\draw [line width=1pt,color=black] (12,-20)-- (12,-24);
		\draw [line width=1pt,color=black] (12,-24)-- (14,-25);
		\draw [line width=1pt,color=black] (14,-25)-- (16,-24);
		\draw [line width=1pt,color=black] (16,-24)-- (12,-20);
		
		\draw [fill=black] (12,-22) circle (2.5pt);
		\draw [fill=black] (12,-24) circle (2.5pt);
		\draw [fill=black] (14,-22) circle (2.5pt);
		\draw [fill=black] (14,-24) circle (2.5pt);
		\draw [fill=black] (12,-20) circle (2.5pt);
		\draw [fill=black] (16,-24) circle (2.5pt);
		\draw [fill=black] (14,-25) circle (2.5pt);

		\draw[color=black] (0,-27.2) node[fill=white] {$b(P)=6$ };
		
		\draw [line width=1pt,color=black] (0,-28)-- (0,-31);
		\draw [line width=1pt,color=black] (0,-31)-- (4,-30);
		\draw [line width=1pt,color=black] (4,-30)-- (6,-28);
		\draw [line width=1pt,color=black] (6,-28)-- (0,-28);
		
		\draw [fill=black] (0,-28) circle (2.5pt);
		\draw [fill=black] (0,-30) circle (2.5pt);
		\draw [fill=black] (2,-28) circle (2.5pt);
		\draw [fill=black] (2,-30) circle (2.5pt);
		\draw [fill=black] (4,-28) circle (2.5pt);
		\draw [fill=black] (4,-30) circle (2.5pt);
		\draw [fill=black] (6,-28) circle (2.5pt);
		\draw [fill=black] (0,-31) circle (2.5pt);

		\draw [line width=1pt,color=black] (8,-28)-- (8,-30);
		\draw [line width=1pt,color=black] (8,-30)-- (9,-31);
		\draw [line width=1pt,color=black] (9,-31)-- (12,-30);
		\draw [line width=1pt,color=black] (12,-30)-- (14,-28);
		\draw [line width=1pt,color=black] (14,-28)-- (8,-28);
		
		\draw [fill=black] (8,-28) circle (2.5pt);
		\draw [fill=black] (8,-30) circle (2.5pt);
		\draw [fill=black] (10,-28) circle (2.5pt);
		\draw [fill=black] (10,-30) circle (2.5pt);
		\draw [fill=black] (12,-28) circle (2.5pt);
		\draw [fill=black] (12,-30) circle (2.5pt);
		\draw [fill=black] (14,-28) circle (2.5pt);
		\draw [fill=black] (9,-31) circle (2.5pt);

		\draw [line width=1pt,color=black] (13,-31)-- (14,-34);
		\draw [line width=1pt,color=black] (14,-34)-- (22,-34);
		\draw [line width=1pt,color=black] (22,-34)-- (13,-31);
		
		\draw [fill=black] (14,-32) circle (2.5pt);
		\draw [fill=black] (16,-32) circle (2.5pt);
		\draw [fill=black] (14,-34) circle (2.5pt);
		\draw [fill=black] (16,-34) circle (2.5pt);
		\draw [fill=black] (18,-34) circle (2.5pt);
		\draw [fill=black] (20,-34) circle (2.5pt);
		\draw [fill=black] (22,-34) circle (2.5pt);
		\draw [fill=black] (13,-31) circle (2.5pt);

		\draw[color=black] (0,-35.2) node[fill=white] {$b(P)=7$ };
		
		\draw [line width=1pt,color=black] (0,-36)-- (0,-39);
		\draw [line width=1pt,color=black] (0,-39)-- (4,-38);
		\draw [line width=1pt,color=black] (4,-38)-- (8,-36);
		\draw [line width=1pt,color=black] (8,-36)-- (0,-36);
		
		\draw [fill=black] (0,-36) circle (2.5pt);
		\draw [fill=black] (0,-38) circle (2.5pt);
		\draw [fill=black] (2,-36) circle (2.5pt);
		\draw [fill=black] (2,-38) circle (2.5pt);
		\draw [fill=black] (4,-36) circle (2.5pt);
		\draw [fill=black] (4,-38) circle (2.5pt);
		\draw [fill=black] (6,-36) circle (2.5pt);
		\draw [fill=black] (8,-36) circle (2.5pt);
		\draw [fill=black] (0,-39) circle (2.5pt);

		\draw [line width=1pt,color=black] (10,-36)-- (10,-38);
		\draw [line width=1pt,color=black] (10,-38)-- (11,-39);
		\draw [line width=1pt,color=black] (11,-39)-- (14,-38);
		\draw [line width=1pt,color=black] (14,-38)-- (18,-36);
		\draw [line width=1pt,color=black] (18,-36)-- (10,-36);
		
		\draw [fill=black] (10,-36) circle (2.5pt);
		\draw [fill=black] (10,-38) circle (2.5pt);
		\draw [fill=black] (12,-36) circle (2.5pt);
		\draw [fill=black] (12,-38) circle (2.5pt);
		\draw [fill=black] (14,-36) circle (2.5pt);
		\draw [fill=black] (14,-38) circle (2.5pt);
		\draw [fill=black] (16,-36) circle (2.5pt);
		\draw [fill=black] (18,-36) circle (2.5pt);
		\draw [fill=black] (11,-39) circle (2.5pt);

		\draw[color=black] (0,-41.2) node[fill=white] {$b(P)=8$ };
		
		\draw [line width=1pt,color=black] (0,-42)-- (0,-45);
		\draw [line width=1pt,color=black] (0,-45)-- (4,-44);
		\draw [line width=1pt,color=black] (4,-44)-- (10,-42);
		\draw [line width=1pt,color=black] (10,-42)-- (0,-42);
		
		\draw [fill=black] (0,-42) circle (2.5pt);
		\draw [fill=black] (0,-44) circle (2.5pt);
		\draw [fill=black] (2,-42) circle (2.5pt);
		\draw [fill=black] (2,-44) circle (2.5pt);
		\draw [fill=black] (4,-42) circle (2.5pt);
		\draw [fill=black] (4,-44) circle (2.5pt);
		\draw [fill=black] (6,-42) circle (2.5pt);
		\draw [fill=black] (8,-42) circle (2.5pt);
		\draw [fill=black] (10,-42) circle (2.5pt);
		\draw [fill=black] (0,-45) circle (2.5pt);

		\draw[color=black] (13,-41.2) node[fill=white] {$b(P)=9$ };
		
		\draw [line width=1pt,color=black] (22,-41)-- (10,-44);
		\draw [line width=1pt,color=black] (10,-44)-- (22,-44);
		\draw [line width=1pt,color=black] (22,-44)-- (22, -41);
		
		\draw [fill=black] (18,-42) circle (2.5pt);
		\draw [fill=black] (20,-42) circle (2.5pt);
		\draw [fill=black] (22,-42) circle (2.5pt);
		\draw [fill=black] (10,-44) circle (2.5pt);
		\draw [fill=black] (12,-44) circle (2.5pt);
		\draw [fill=black] (14,-44) circle (2.5pt);
		\draw [fill=black] (16,-44) circle (2.5pt);
		\draw [fill=black] (18,-44) circle (2.5pt);
		\draw [fill=black] (20,-44) circle (2.5pt);
		\draw [fill=black] (22,-44) circle (2.5pt);
		\draw [fill=black] (22,-41) circle (2.5pt);

	\end{tikzpicture}
	\caption{\label{1_interior_lattice_point}All 30 pseudo-integral polygons with denominator $2$ and exactly $1$ interior lattice points up to affine unimodular equivalence.}
\end{figure}
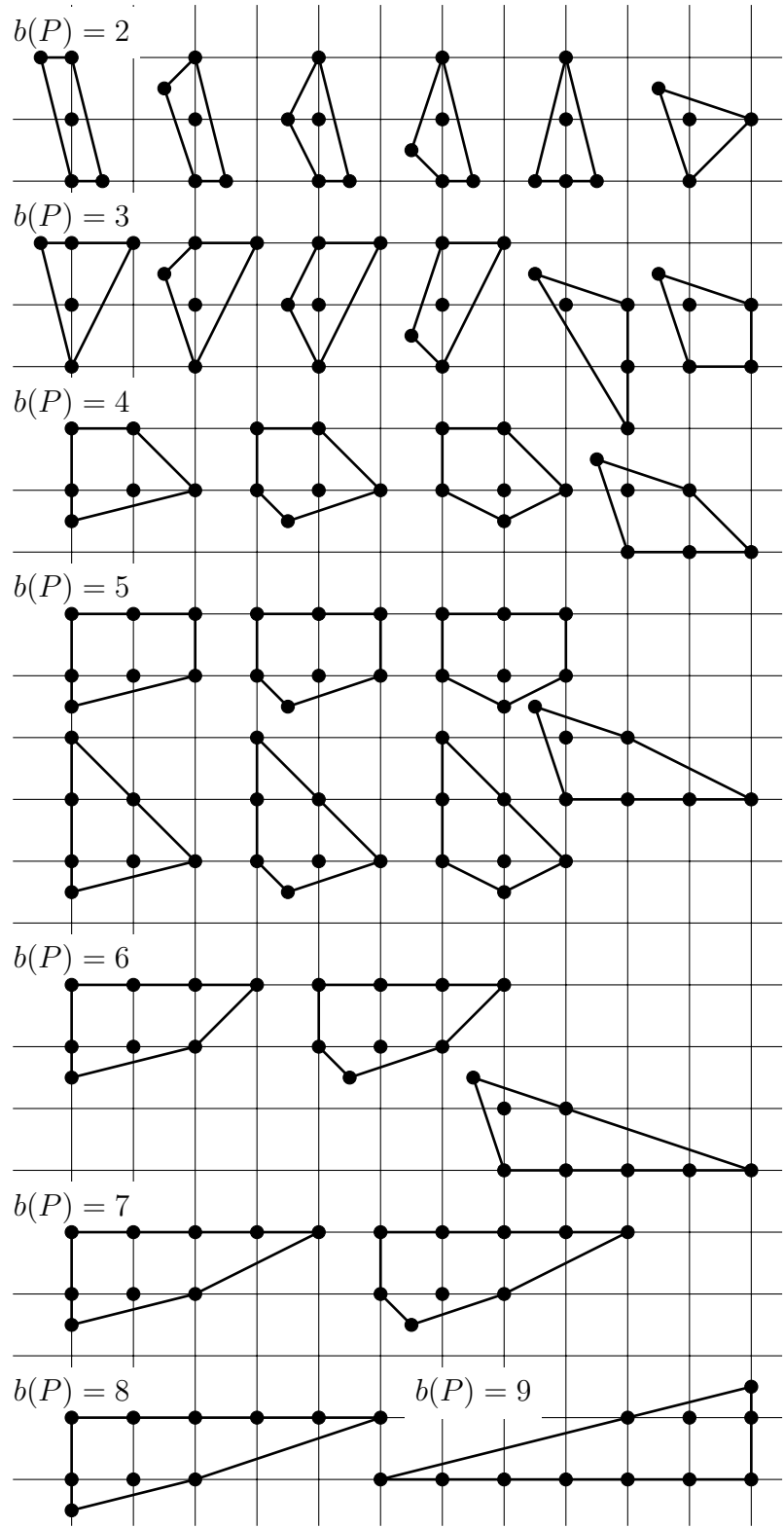

We get all lattice polygons without interior lattice points, $b(P)+1$ boundary lattice points and one edge with exactly $3$ lattice points or two adjacent edges with exactly $2$ lattice points by \ref{hollow_lattice_polygons}. So we obtain that $\inthull{P}$ is either a triangle affine unimodular equivalent to
\begin{align*}
\conv{(1,0),(b(P),0), (1,1)} \text{ or } \conv{(0,0), (2,0), (0,2)} 
\end{align*}
or it is a quadrilateral affine unimodular equivalent to
\begin{align*}
&\conv{(1,0), (b(P)-1,0), (2,1), (1,1)} \text{ or }\\ &\conv{(1,0), (b(P)-2,0), (3,1), (1,1)},
\end{align*}
where we should restrict to $b(P)\geq 3$ in the first case and $b(P)\geq 5$ in the second case to really get a quadrilateral and have each polygon only once.
	
If the closure of $P\setminus \inthull{P}$ in $\R^2$ is a triangle affine unimodular equivalent to $\conv{(0,0), (2,0), (0,\frac{1}{2})}$, then $P$ is either affine unimodular equivalent to a subpolygon of $\conv{(0,-1), (6,-1), (0,\frac{1}{2})}$ with half-integral vertex in $(0,\frac{1}{2})$ and $(0,0), (2,0)\in \partial P$, or has $\inthull{P}\cong \conv{(0,0), (2,0), (0,2)}$ and is therefore affine unimodular equivalent to the convex hull of $(0,0), (2,0), (0,2)$ and a half-integral vertex which can be chosen from the set  $\{(\frac{l}{2},-\frac{1}{2}) \mid l\in\Z_{\geq 0}, l\leq 5\}$. This is due to the convexity of $P$ and the possibilities for $\inthull{P}$ listed above. Thus, we have only a few possibilities for $P$ and can check all of them for affine unimodular equivalence.
	
If the closure of $P\setminus \inthull{P}$ in $\R^2$ is the union of two triangles sharing an edge, while each of them is affine unimodular equivalent to $\conv{(0,0), (1,0), (0,\frac{1}{2})}$, we must have two edges with exactly $2$ lattice points for $\inthull{P}$, and so we can assume by the above possibilities for $\inthull{P}$ and by a suitable affine unimodular transformation that we have either
\begin{align*}
\inthull{P}=&\conv{(1,0),(b(P),0), (1,1)} \text{ or }\\
\inthull{P}=&\conv{(1,0), (b(P)-1,0), (2,1), (1,1)},\ b(P)\geq 3.
\end{align*}

In the first case  $\inthull{P}=\conv{(1,0),(b(P),0), (1,1)}$, we can assume that the second coordinate of the half integral vertex is at least $\frac{3}{2}$ and get 
\begin{align*}
\frac{3}{2}\cdot \frac{b(P)-1}{2}\leq \area{P}=\frac{b(P)}{2}
\end{align*}
and therefore $b(P)\leq 3$. So we have either $P=\conv{(1,0),(2,0), (\frac{1}{2},2)}$ or $P=\conv{(1,0),(3,0), (\frac{1}{2},\frac{3}{2})}$.

In the second case $\inthull{P}=\conv{(1,0), (b(P)-1,0), (2,1), (1,1)}$ and $b(P)\geq 3$, we can assume with the help of the symmetry of $P$ that the half-integral vertex of $P$ is $(\frac{1}{2},\frac{3}{2})$, since it must be a point with lattice distance $\frac{1}{2}$ to two edges with $2$ lattice points. Due to the convexity of $P$, we get $P$ as a subpolygon of $\conv{(\frac{1}{2},\frac{3}{2}),(1,0), (5,0)}$ with half-integral vertex $(\frac{1}{2},\frac{3}{2})$. Thus $P$ is the polygon $\conv{(\frac{1}{2},\frac{3}{2}),(1,0), (b(P)-1,0), (2,1)}$.
	
All in all, we get up to affine unimodular equivalence the polygons shown in figure \ref{1_interior_lattice_point}, and these polygons are all pseudo-integral with denominator $2$, since they obey Pick's formula and have a lattice point on every edge.
\end{proof}

\section{Pseudo-integral polygons with exactly one interior lattice point and duals of LDP polygons}

To understand the case of half-integral pseudo-integral polygons $P\subseteq \R^2$ with exactly one interior lattice point on a deeper level, it is helpful to note that their duals belong to a special class of lattice polygons, namely \textit{LDP polygons}, which we want to introduce at the beginning of this section.

Period-collapse for duals of LDP polygons is well understood via so-called \textit{mutations}, i.e. some combinatorial transformations introduced in \cite{ACGK12} to study mirror symmetry for Fano manifolds (\cite{CCG13}, \cite{ACC16}). Mutations are discussed in the context of period collapse especially in \cite{KW18}. We do not use these tools here and work purely combinatorially.

\begin{defi}
Let $P\subseteq \R^n$ be a rational polytope, $0\in \interior{P}$, and we denote by $\left<\cdot,\cdot\right>\colon \R^n\times \R^n\to \R$ the standard inner product of $\R^n$. If we identify $(\R^n)^\ast$ with $\R^n$ using the inner product, then we define the \textit{dual polytope} $P^\ast \subseteq (\R^n)^\ast\cong \R^n$ of $P$ by
\begin{align*}
P^\ast :=\{y \in \R^n \mid \left<y,x\right> \geq -1 \ \forall x \in P \}.
\end{align*} 
\end{defi}

\begin{defi}
Let $H_{n,b}\subseteq \R^d$ be a rational hyperplane defined by a dual lattice vector $\left<\cdot,n\right> \in (\Z^d)^\ast$ and $b\in \Z$ as 
\begin{align*}
H_{n,b}=\left\{x\in \R^n \mid \left<x,n\right>=b \right\}.
\end{align*}
Then the \textit{lattice distance} $\ldist{H_{n,b}}{l}$ of a rational point $l \in \Q^d$ to the rational hyperplane $H_{n,b}$ is defined as
\begin{align*}
\ldist{H_{n,b}}{l}:=|\left<l,n\right>-b|.
\end{align*}
The lattice distance to a rational polytope of codimension 1 is defined as the lattice distance to the affine hull of the polytope.
\end{defi}

In general, $P^*$ is just a rational polytope, and we cannot even expect it to be a lattice polytope if $P$ is a lattice polytope. However, if $P^*$ is a lattice polytope, we have a theorem of Hibi which describes the behavior of the lattice points of $P$.

\begin{thm}[\cite{Hib92}]\label{Hibi}
Let $P\subseteq \R^n$ be a rational polytope with $0\in \mathrm{int}(P)$. Then the following conditions are equivalent
\begin{itemize}
\item $P^\ast$ is a lattice polytope.
\item For every facet $F\preceq P$ there is $a\in \Z_{\geq 1}$, so that $\ldist{F}{0}=\frac{1}{a}$.
\item $l(kP)=i((k+1)P)$ for all $k\in \Z_{\geq 0}$.
\end{itemize}
\end{thm}

In particular, we see from the theorem that $P$ has only $0$ as an interior lattice point, if $P^\ast$ is a lattice polytope. We can force the rational polytope $P$ to behave more like a lattice polytope by requiring that $P^*$ has primitive vertices. This results in the following definitions.

\begin{defi}
Let $P\subseteq \R^n$ be a lattice polytope with $0\in \mathrm{int}(P)$.
If the coordinates of every vertex of $P$ are relatively prime, then $P$ is called a \textit{Fano polytope}. The number $\denom{P^\ast}$ is called the \textit{Gorenstein index $i_g$ of $P$}. If $P^\ast$ is also a lattice polytope, then $P$ is called a \textit{reflexive polytope}, which is the same as a Fano polytope of Gorenstein index $1$. Fano polytopes of dimension $2$ are also called \textit{LDP polygons}.
\end{defi}

Fano polytopes correspond to toric Fano varieties via toric geometry. For a survey of Fano polytopes see \cite{KN12}. Reflexive polytopes are especially known for their role in combinatorial mirror symmetry, since they provide a way to construct mirror pairs of Calabi-Yau varieties (\cite{Bat94}).

If $P^\ast$ is a Fano polytope, we get the following corollary of Hibi's theorem.

\begin{coro}\label{Hibi_corollary}
Let $P\subseteq \R^n$ be a rational polygon with $0\in \mathrm{int}(P)$. Then the following conditions are equivalent
\begin{itemize}
\item $P^\ast$ is a Fano polytope.
\item For each facet $F\preceq P$ we have $\ldist{F}{0}=1$.
\item $l(kP)=i((k+1)P)$ for all $k\in \Z_{\geq 0}$ and the affine hull of every facet of $P$ contains lattice points.
\end{itemize}
\end{coro}

We now focus on the $2$ dimensional pseudo-integral case and in particular get all pseudo-integral polygons with exactly one interior lattice point as duals of LDP polygons.

\begin{coro}\label{pseudo_integral_1_interior_point}\label{Hibi_corollary2}
Let $P\subseteq \R^2$ be a rational pseudo-integral polygon with $0\in \mathrm{int}(P)$. Then the following conditions are equivalent
\begin{itemize}
\item $P^\ast$ is an LDP polygon.
\item For each edge $E\preceq P$ we have $\ldist{E}{0}=1$.
\item $l(kP)=i((k+1)P)$ for all $k\in \Z_{\geq 0}$ and the affine hull of every edge of $P$ contains lattice points.
\item $\ehr{P}{t}=\area{P}t^2+\area{P}t+1$.
\item $\mathrm{int}(P)\cap \Z^2=\{(0,0)\}$.
\end{itemize}
\end{coro}
\begin{proof}
Since $P$ is pseudo-integral, we have $\ehr{P}{t}=\area{P}t^2+\frac{b(P)}{2}t+1$ and we have Pick's formula for $P$ by \ref{pseudo_crit_general}.

We already know that the first three conditions are equivalent by \ref{Hibi_corollary} and they imply $\mathrm{int}(P)\cap \Z^2=\{(0,0)\}$, since we cannot have more interior lattice points, if every edge $E\preceq P$ has $\ldist{E}{0}=1$. With Pick's formula we get from $i(P)=1$ that $b(P)=2\area{P}-2i(P)+2=2\area{P}$ and therefore we have the Ehrhart polynomial $\ehr{P}{t}=\area{P}t^2+\area{P}t+1$.

Conversely, if $\ehr{P}{t}=\area{P}t^2+\area{P}t+1$ and therefore $2\area{P}=b(P)$, we get for all $k\in \Z_{\geq 0}$ by Ehrhart-Macdonald reciprocity \ref{Ehrhart-Macdonald}
\begin{align*}
l(kP)=&\area{P}k^2+\area{P}k+1=\area{P}(k+1)^2-\area{P}(k+1)+1\\=&i((k+1)P).
\end{align*}
Thus we also have $\ldist{E}{0}\leq 1$ for all edges $E\preceq P$ by \ref{Hibi} and so we get with $\denom{P}\cdot b(P)=b(\denom{P}\cdot P)$ from \ref{pseudo_crit_general} and using the lattice length of an edge $\text{length}(E)=\frac{l(\denom{P}\cdot E)-1}{2\denom{P}}$ that
\begin{align*}
\area{P}=\sum_{E\preceq P}\ldist{E}{0}\cdot \frac{\text{length}(E)}{2}=\sum_{E\preceq P}\ldist{E}{0}\cdot \frac{l(\denom{P}\cdot E)-1}{2\denom{P}}\leq  \frac{b(P)}{2}.
\end{align*}
Since $2\area{P}=b(P)$ we must have equality and so every edge must have lattice distance $1$ from $0$.
\end{proof}

We know all LDP polygons up to Gorenstein index $i_g=17$ by \cite{KKN10}. Thus, by \ref{pseudo_integral_1_interior_point}, the pseudo-integral polygons with $i(P)=1$ and denominator at most $17$ are up to affine unimodular equivalence among the duals of these polygons. We can test whether the duals are really pseudo-integral and get the following result, which in particular gives a second classification of pseudo-integral polygons with denominator $2$ and exactly one interior lattice point (for the complete list of the polygons in the result see \cite{Boh24a}).

\begin{classify}\label{classification_dual_fano}
Up to affine unimodular equivalence there are exactly $320$ pseudo-integral polygons with $i(P)=1$ and denominator $d$ at most $17$. They have at most $9$ boundary lattice points and there are no examples for $d\in \{7,13,17\}$.

The number of polygons depending on the denominator and the number of boundary lattice points is given in the following table, where the $b$-th entry in the vector in the third column gives the number of pseudo-integral polygons with $i(P)=1$ and $b$ boundary lattice points up to affine unimodular equivalence.

\begin{center}
\begin{tabular}{c|c|c}
$d$&$\#$ LDP polygons with $i_g=d$ &$\#$ pseudo-integral polygons \\\hline\hline
$1$&$16$&$(0,0,1,3,2,4,2,3,1)$\\
$2$&$30$&$(0,6,6,4,7,3,2,1,1)$\\
$3$&$99$&$(1,6,1,3,2,1,1,1,0)$\\
$4$&$91$&$(1,1,3,2,1,1,1,1,0)$\\
$5$&$250$&$(1,1,2,2,1,2,1,1,1)$\\
$6$&$379$&$(11,10,8,12,10,4,1,0)$\\
$7$&$429$&$(0,0,0,0,0,0,0,0,0)$\\
$8$&$307$&$(0,0,0,0,1,0,0,0,0)$\\
$9$&$690$&$(0,0,0,1,0,0,1,0,0)$\\
$10$&$916$&$(12,2,6,4,3,1,1,1,1)$\\
$11$&$939$&$(0,0,0,1,0,0,0,1,0)$\\
$12$&$1279$&$(12,10,14,15,6,5,3,1,0)$\\
$13$&$1142$&$(0,0,0,0,0,0,0,0,0)$\\
$14$&$1545$&$(0,1,1,1,2,0,0,0,0)$\\
$15$&$4312$&$(2,13,11,21,9,6,5,4,0)$\\
$16$&$1030$&$(0,0,0,0,0,0,1,0,0)$\\
$17$&$1892$&$(0,0,0,0,0,0,0,0,0)$\\
\end{tabular}
\end{center}
\end{classify}

Even more is known about LDP triangles. The manual classifications for $i_g=2$ in \cite{Dai06} and $i_g=3$ in \cite{Dai09} have recently been extended to $i_g=200$ in the toric part of \cite{HHHS25} and to $i_g=1000$ in \cite{Bae25}. We have done the computations for the LDP triangles up to $i_g=1000$ and get the following result (for the complete list of the polygons in the result see \cite{Boh24a}).

\begin{classify}
There are up to affine unimodular equivalence exactly $135$ pseudo-integral triangles with $i(P)=1$ and denominator at most $1000$. All triangles have between $1$ and $9$ boundary lattice points, and there is no example for $b(P)=7$. Depending on the number of boundary lattice points, these triangles are distributed as shown in the following table.

\begin{center}
\begin{tabular}{c||c|c|c|c|c|c|c|c|c}
$b(P)$&$1$&$2$&$3$&$4$&$5$&$6$&$7$&$8$&$9$\\\hline
\# pseudo-integral triangles&$39$&$26$&$19$&$8$&$15$&$14$&$0$&$8$&$6$
\end{tabular}
\end{center}
\end{classify}

\begin{rem}
Experimental data lead to several questions about pseudo-integral polygons with exactly one interior lattice point. Are there always at most $9$ boundary lattice points?
A positive answer to this question for triangles was announced in a talk by Tyrrell B. McAllister \cite{MAW22}. Are there infinitely many examples that are triangles? Are there infinitely many examples for a given number of boundary lattice points? Is there an example which is a triangle and has exactly $7$ boundary lattice points?

The article \cite{MAW24}, that was written after the first preprint of this article, provides answers to the questions about triangles. There are pseudo-integral triangles with one interior lattice point and $b$ boundary lattice points if and only if $1\leq b \leq 9, b\neq 7$. Moreover, there are infinitely many of them in each allowed case. These triangles can be considered a generalization of the \textit{Markov triangles} mentioned in, for example, \cite[3.8]{KW18}. For more connections to algebraic geometry, see also \cite{HK24}.

We leave the other questions open and concentrate now on the half-integral case. In the half-integral case we see that every dual of an LDP polygon is pseudo-integral. This gives a new possibility to classify LDP polygons with $i_g=2$ as duals of pseudo-integral polygons with denominator $2$ and exactly one interior lattice point. Therefore, we give a proof of this observation without using the above data in the following theorem.
\end{rem}

\begin{thm}
Let $P\subseteq \R^2$ be a half-integral polygon with $0\in \mathrm{int}(P)$. Then the following conditions are equivalent
\begin{itemize}
\item $P^\ast$ is an LDP polygon.
\item Each edge of $P$ has lattice distance $1$ from $0$.
\item $|kP\cap \Z^2|=|\mathrm{int}((k+1)P)\cap \Z^2|$ for all $k\geq 0$ and the affine hull of every edge of $P$ contains lattice points.
\item $\ehr{P}{t}=\area{P}t^2+\area{P}t+1$.
\item $\mathrm{int}(P)\cap \Z^2=\{(0,0)\}$ and $P$ is pseudo-integral.
\end{itemize}
\end{thm}
\begin{proof}
By \ref{Hibi_corollary} and \ref{Hibi_corollary2} it is sufficient to show that the first three conditions imply that $P$ is pseudo-integral. If the affine hull of every edge of $P$ contains lattice points, then we have lattice points on every edge since $\denom{P}=2$ and we get $b(2P)=2b(P)$. Thus, with \ref{pseudo_crit} it suffices to show that the first three conditions imply that $P$ obeys Pick's formula.

Since every edge of $P$ has lattice distance $1$ from $0$, every edge of $2P$ has lattice distance $2$ from $0$. Thus, $b(2P)=\area{2P}=4\cdot \area{P}$ and we get
\begin{align*}
i(P)+\frac{b(P)}{2}-1=1+\frac{b(2P)}{4}-1=\area{P}.
\end{align*}
\end{proof}

\begin{rem}
We can also see with \cite[3.4]{KW18} that all duals of LDP polygons with Gorensteinindex $2$ are pseudo-integral. This is because all cyclic quotient singularities of local index $2$ are $T$-singularities. We see this combinatorially because every lattice polygon with lattice height $2$ has an even number of boundary lattice points due to Pick's formula. 
	
Do not forget, that the theorem is not correct for polygons with higher denominators, as we saw in \ref{classification_dual_fano}. For example, we have $\Delta:=\conv{(1,0),(0,-\frac{1}{3}),(-1,2)}$ with dual $\Delta^\ast=\conv{(7,3),(-1,-1),(-1,3)}$, which is an LDP polygon of Gorenstein index $3$, but the period sequence of $\Delta$ is $(3,1,1)$.
\end{rem}

We end this section with a formula for the number of boundary lattice points of $P$ in the half-integral case, depending on the number of interior and boundary lattice points of the dual LDP polygon. This generalizes the equation $b(P)+b(P^\ast)=12$, which we have for reflexive polygons (\cite{PRV00}). We therefore use the following even more general theorem for LDP polygons, which is itself a combinatorial interpretation of the \textit{stringy Libgober-Wood identity} from \cite{LW90}. A purely combinatorial proof of this general theorem was given in \cite{BHSdW23}.

\begin{thm}[\protect{\cite[Corollary 4.5]{BS17}}]\label{area_LDP_polygon} Let $P\subseteq \R^2$ be dual to an LDP polygon. Then
\begin{align*}
2\cdot(\area{P}+\area{P^\ast})=12 \sum_{l\in P^\ast\cap \Z^2} (\kappa_{P^*}(l)+1)^2
\end{align*}
with
\begin{align*}
\kappa_{P^*}(l)=-\min \{\lambda \in \R_{\geq 0} \mid l\in \lambda P^\ast\}.
\end{align*}
In particular, $2(\area{P}+\area{P^\ast})\geq 12$ with equality if and only if $P$ is reflexive.
\end{thm}

\begin{coro}
Let $P\subseteq\R^2$ be a pseudo-integral polygon with $\mathrm{int}(P)\cap \Z^2=\{(0,0)\}$ and $\denom{P}\leq 2$. Then
\begin{align*}
b(P)+b(P^\ast)=12+(i(P^\ast)-1).
\end{align*}
\end{coro}
\begin{proof}
Since $P$ is pseudo-integral with $i(P)=1$, we have $2\cdot \area{P}=b(P)$ by Pick's formula. Since $P^\ast$ is a lattice polygon by \ref{Hibi_corollary2}, we have by Pick's theorem $2\cdot\area{P^\ast}=2\cdot i(P^\ast)+b(P^\ast)-2$. So we get with \ref{area_LDP_polygon}
\begin{align*}
b(P)+b(P^\ast)=&2\cdot(\area{P}+\area{P^\ast})-2\cdot i(P^\ast)+2\\
=&-2\cdot i(P^\ast)+2+12 \sum_{l\in P^\ast\cap \Z^2} (\kappa_{P^*}(l)+1)^2.
\end{align*}
Note that $\kappa_{P^*}(l)=-1$ if $l\in \partial P^\ast$ and $\kappa_{P^*}((0,0))=0$. Moreover, since $P$ is half-integral, $P^\ast$ has Gorenstein index $i_g\leq 2$ and so every interior lattice point $l$ of $P^\ast$ different from $(0,0)$ has $\kappa_{P^*}(l)=-\frac{1}{2}$. So we get
\begin{align*}
b(P)+b(P^\ast)=&-2\cdot i(P^\ast)+2+12 \sum_{l\in P^\ast\cap \Z^2} (\kappa_{P^*}(l)+1)^2\\
=&12+\sum_{l\in \mathrm{int}(P^\ast)\cap \Z^2, l \neq (0,0)} \left(12\cdot (-\frac{1}{2}+1)^2-2\right)\\
=&12+(i(P^\ast)-1).
\end{align*}
\end{proof}

\section{A two parameter family of half-integral pseudo-integral polygons with $i\geq 1$ interior and $3\leq b\leq 2i+7$ boundary lattice points}

In the last sections we classified all pseudo-integral polygons of denominator $2$ with at most $1$ interior lattice point, and we gave an example of such a polygon for any given positive number of interior lattice points if $b(P)=2$.

In the following proposition we give an example of a pseudo-integral polygon of denominator $2$ with $i(P)\in \Z_{\geq 1}, 3\leq b(P)\leq 2i(P)+7$.

\begin{prop}
Let be $i, b\in \Z_{\geq 1}$ with $3 \leq b\leq 2i+7$. Then
\begin{align*}
P_{2,(i,b)}:=\conv{\left(0,\frac{3}{2}\right), (0,1), (2i+7-b,0), (2i+4,0), \left(2i+2,\frac{1}{2}\right)}\subseteq \R^2
\end{align*}
is a pseudo-integral polygon with $\denom{P}=2$, $i(P_{2,(i,b)})=i$ and $b(P_{2,(i,b)})=b$.
\end{prop}
\begin{proof}
We have $\denom{P_{2,(i,b)}}=2$,
\begin{align*}
|\mathrm{int}(P_{2,(i,b)})\cap \Z^2|=|\{(k,1) \mid k\in \Z, 1\leq k\leq i\}|=i
\end{align*}
and
\begin{align*}
|\partial P_{2,(i,b)}\cap \Z^2|=&|\{(0,1),(i+1,1)\}\cup \{(k,0) \mid k\in \Z, 2i+7-b\leq k\leq 2i+4\}|\\=&2+b-2=b.
\end{align*}
Since 
\begin{align*}
\area{P_{2,(i,b)}}=&\area{\conv{\left(0,\frac{3}{2}\right), (0,0), (3i+3,0)}}\\
&\qquad -\area{\conv{\left(0,1\right), (0,0), (2i+7-b,0)}}\\
&\qquad -\area{\conv{\left(2i+2,\frac{1}{2}\right), (2i+4,0), (3i+3,0)}}\\
=&\frac{3(3i+3)}{4}-\frac{2i+7-b}{2}-\frac{i-1}{4}\\
=&i+\frac{b}{2}-1,
\end{align*}
the polygon $P_{2,(i,b)}$ satisfies Pick's formula and we have lattice points on every edge of $P_{2,(i,b)}$, so with \ref{pseudo_crit} we get also that $P_{2,(i,b)}$ is pseudo-integral. 
\end{proof}

\begin{figure}[H]
	\begin{tikzpicture}[x=0.6cm,y=0.6cm]
		\draw[step=2.0,black,thin,xshift=0cm,yshift=0cm] (-1,3.5) grid (17,-12.5);
		
		\draw [line width=1pt,color=black] (0,3)-- (0,2);
		\draw [line width=1pt,color=black] (0,2)-- (16,0);
		\draw [line width=1pt,color=black] (16,0)-- (12,1);
		\draw [line width=1pt,color=black] (12,1)-- (0,3);
		
		\draw [fill=black] (0,2) circle (2.5pt);
		\draw [fill=black] (2,2) circle (2.5pt);
		\draw [fill=black] (4,2) circle (2.5pt);
		\draw [fill=black] (6,2) circle (2.5pt);
		\draw [fill=black] (0,3) circle (2.5pt);
		\draw [fill=black] (12,1) circle (2.5pt);
		\draw [fill=black] (16,0) circle (2.5pt);
		\draw [fill=black] (16,0) circle (2.5pt);
		
		\draw[color=black] (15,1) node {$P_{2,(2,3)}$ };

		\draw [line width=1pt,color=black] (0,-1)-- (0,-2);
		\draw [line width=1pt,color=black] (0,-2)-- (14,-4);
		\draw [line width=1pt,color=black] (14,-4)-- (16,-4);
		\draw [line width=1pt,color=black] (16,-4)-- (12,-3);
		\draw [line width=1pt,color=black] (12,-3)-- (0,-1);
		
		\draw [fill=black] (0,-2) circle (2.5pt);
		\draw [fill=black] (2,-2) circle (2.5pt);
		\draw [fill=black] (4,-2) circle (2.5pt);
		\draw [fill=black] (6,-2) circle (2.5pt);
		\draw [fill=black] (0,-1) circle (2.5pt);
		\draw [fill=black] (12,-3) circle (2.5pt);
		\draw [fill=black] (14,-4) circle (2.5pt);
		\draw [fill=black] (16,-4) circle (2.5pt);
		
		\draw[color=black] (15,-3) node {$P_{2,(2,4)}$ };

		\draw [line width=1pt,color=black] (0,-5)-- (0,-6);
		\draw [line width=1pt,color=black] (0,-6)-- (2,-8);
		\draw [line width=1pt,color=black] (2,-8)-- (16,-8);
		\draw [line width=1pt,color=black] (16,-8)-- (12,-7);
		\draw [line width=1pt,color=black] (12,-7)-- (0,-5);
		
		\draw [fill=black] (0,-6) circle (2.5pt);
		\draw [fill=black] (2,-6) circle (2.5pt);
		\draw [fill=black] (4,-6) circle (2.5pt);
		\draw [fill=black] (6,-6) circle (2.5pt);
		\draw [fill=black] (0,-5) circle (2.5pt);
		\draw [fill=black] (12,-7) circle (2.5pt);
		\draw [fill=black] (2,-8) circle (2.5pt);
		\draw [fill=black] (4,-8) circle (2.5pt);
		\draw [fill=black] (6,-8) circle (2.5pt);
		\draw [fill=black] (8,-8) circle (2.5pt);
		\draw [fill=black] (10,-8) circle (2.5pt);
		\draw [fill=black] (12,-8) circle (2.5pt);
		\draw [fill=black] (14,-8) circle (2.5pt);
		\draw [fill=black] (16,-8) circle (2.5pt);
		
		\draw[color=black] (15,-7) node {$P_{2,(2,10)}$ };

		\draw [line width=1pt,color=black] (0,-9)-- (0,-10);
		\draw [line width=1pt,color=black] (0,-10)-- (0,-12);
		\draw [line width=1pt,color=black] (0,-12)-- (16,-12);
		\draw [line width=1pt,color=black] (16,-12)-- (12,-11);
		\draw [line width=1pt,color=black] (12,-11)-- (0,-9);
		
		\draw [fill=black] (0,-10) circle (2.5pt);
		\draw [fill=black] (2,-10) circle (2.5pt);
		\draw [fill=black] (4,-10) circle (2.5pt);
		\draw [fill=black] (6,-10) circle (2.5pt);
		\draw [fill=black] (0,-9) circle (2.5pt);
		\draw [fill=black] (12,-11) circle (2.5pt);
		\draw [fill=black] (0,-12) circle (2.5pt);
		\draw [fill=black] (2,-12) circle (2.5pt);
		\draw [fill=black] (4,-12) circle (2.5pt);
		\draw [fill=black] (6,-12) circle (2.5pt);
		\draw [fill=black] (8,-12) circle (2.5pt);
		\draw [fill=black] (10,-12) circle (2.5pt);
		\draw [fill=black] (12,-12) circle (2.5pt);
		\draw [fill=black] (14,-12) circle (2.5pt);
		\draw [fill=black] (16,-12) circle (2.5pt);
		
		\draw[color=black] (15,-11) node {$P_{2,(2,11)}$ };
		
	\end{tikzpicture}
	\caption{Some pseudo-integral half-integral polygons with $2$ interior lattice points. We have drawn polygons affine unimodular equivalent to $P_{2,(2,3)}$, $P_{2,(2,4)}$, $P_{2,(2,10)}$ and $P_{2,(2,11)}$. }
\end{figure}
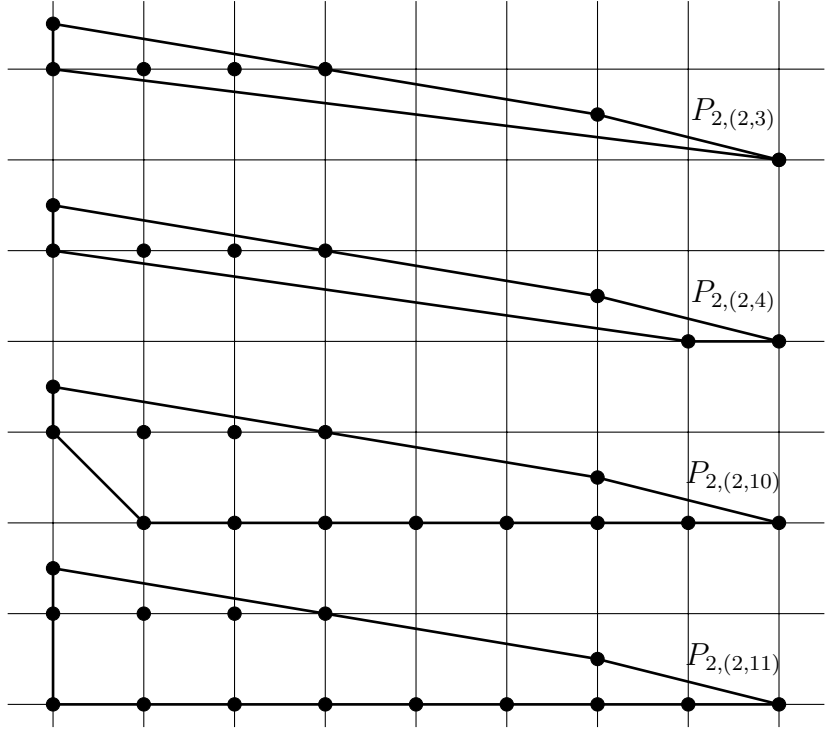

We have seen explicit examples with $b(P)=2i(P)+7$ and $i(P)>1$ in the last proposition, i.e. we have examples that do not obey Scott's inequality \ref{Scott_inequality}. It was an open question in \cite{MAM17} whether such pseudo-integral polygons exist. Can we classify all such polygons, and can we have even more boundary lattice points for pseudo-integral polygons of denominator $2$? The answer to the first question is 'yes', we will do the classification in the next section. For the second question, we will see below that this is not the case, starting with a result that gives us suitable coordinates for half-integral polygons of large area.

\begin{lemma}\label{half_integral_coordinates_for_big_areas}
Let $P\subseteq \R^2$ be a half-integral polygon with $\area{P}\geq 2i(P)+3$.\\
Then $P$ is affine unimodular equivalent to a polygon in $[0,\frac{3}{2}]\times \R$.
If $P\subseteq [0,2]\times \R$ with $P\cap \{0\}\times \R\neq \emptyset$ and $P\cap \{2\}\times \R\neq \emptyset$, then $\area{P}\leq 2i(P)+2$.
\end{lemma}
\begin{proof}
This is a consequence of \cite[3.1, 3.3, 3.5]{Boh23}.
\end{proof}

\begin{thm}\label{upper_bound_boundary_points}
Let $P\subseteq \R^2$ be a half-integral pseudo-integral polygon with $i(P)>0$ and $b(P)\geq 2i(P)+7$. Then $b(P)= 2i(P)+7$.
\end{thm}
\begin{proof}
Suppose $P$ is a half-integral pseudo-integral polygon with $b(P)>2i(P)+7$. Then $\area{P}\geq 2i(P)+3$ by Pick's formula and so $P$ is by \ref{half_integral_coordinates_for_big_areas} affine unimodular equivalent to a half-integral polytope $Q$ in $[0,\frac{3}{2}]\times \R$. Since $P$ has $i(P)$ interior lattice points, we can also assume, with the help of a suitable translation and shearing, that 
\begin{align*}
Q\cap \{1\}\times \R\subseteq \{1\}\times [0,i(P)+1]
\end{align*}
and $(\frac{3}{2},0)$ is a vertex of $Q$. Since $Q$ is pseudo-integral, we have 
\begin{align*}
i(2Q)=4i(P)+b(P)-3
\end{align*}
by \ref{int_bound_points} and therefore $i(2Q)>6i(P)+4$. But since $Q\cap \{1\}\times \R\subseteq \{1\}\times [0,i(P)+1]$ and $(\frac{3}{2},0)$ is a vertex of $Q$, we have
\begin{align*}
\mathrm{int}(2Q)\subseteq \conv{(1,1),(2,1),(2,2i(P)+1),(1,4i(P)+3)}
\end{align*}
and so $i(2Q)\leq 2i(P)+1+4i(P)+3=6i(P)+4$, a contradiction. 
\end{proof}

We summarize our results from the last sections on the possible values of the pair $(i(P),b(P))$ in the following theorem. With \ref{pseudo_crit_general} this theorem gives the proof of our main theorem \ref{main_result}.

\begin{thm}
Let $P\subseteq \R^2$ be a half-integral pseudo-integral polygon that is not a lattice polygon. Then 
\begin{align*}
(i(P),b(P))\in \{(0,3)\} \cup \{(i,b) \mid i, b\in \Z_{\geq 1}, 2\leq b\leq 2i+7\}.
\end{align*}
Conversely, there is a half-integral pseudo-integral polygon $P\subseteq \R^2$, which is not a lattice polygon, with $(i(P),b(P))=(i,b)$ for every 
\begin{align*}
(i,b)\in \{(0,3)\} \cup \{(i,b) \mid i, b\in \Z_{\geq 1}, 2\leq b\leq 2i+7\}.
\end{align*}
\end{thm}

\section{Classification of half-integral pseudo-integral polygons with $i\geq 1$ interior and $b = 2i+7$ boundary lattice points}

For the classification, we need the following lemma about the integer hull of pseudo-integral polygons with many boundary lattice points.

\begin{lemma}\label{hollow_integer_hull}
Let $P\subseteq \R^2$ be a rational polygon with $\denom{P}>1$, $i(P)\geq 1$ and $b(P)\geq 2i(P)+7$. Then $\inthull{P}$ has no interior lattice points.
\end{lemma}
\begin{proof}
The case $i(P)>1, b(P)\geq 2i(P)+7$ follows from \cite[3.1]{MAM17}. Suppose $i(P)=1$, $b(P)\geq 9$ and $i(\inthull{P})>0$, then we have $b(\inthull{P})=b(P)=9$ and $\inthull{P}\cong \conv{(0,3),(0,0),(3,0)}$ and so we get $P=\inthull{P}$, since any larger $P$ would have more interior lattice points. So $P$ has denominator $1$, a contradiction.
\end{proof}

\begin{thm}
$P\subseteq \R^2$ is a pseudo-integral polygon with $\denom{P}=2, i(P)=i$ and $b(P)=2i+7$ if and only if either
\begin{align*}
P\cong P_{2,(i,2i+7),a}:=\conv{\left(0,\frac{3}{2}\right), \left(0,\frac{1}{2}\right), (a,0), (a+2i+4,0), \left(2i+2,\frac{1}{2}\right)}
\end{align*}
for an $a\in \Z, 0\leq a\leq \frac{i-1}{2}$ or 
\begin{align*}
i>1, P\cong P_{2,(i,2i+7),s}:= \conv{\left(0,\frac{3}{2}\right), (0,1), \left(\frac{1}{2},0\right), (2i+5,0), \left(2i+2,\frac{1}{2}\right)} 
\end{align*}
or 
\begin{align*}
i=3, P\cong P_{2,(3,13),s_3}:=\conv{\left(0,\frac{3}{2}\right), (0,1), \left(\frac{1}{2},0\right), \left(\frac{23}{2},0\right), (4,1)}. 
\end{align*}
\end{thm}

\begin{proof}
Let $P\subseteq \R^2$ be a pseudo-integral polygon with $\denom{P}=2, i(P)=i$ and $b(P)=2i(P)+7$. Due to explicit computations in \ref{classification_small_number_interior_points} we can assume $i(P)>6$.

As a first step, we show that we can choose suitable coordinates so that we can realize $P$ as a subpolygon of $\conv{(0,\frac{3}{2}), (0,0), (3i+3)}$. By \ref{hollow_integer_hull} we know that $\inthull{P}$ has no interior lattice points. Since there are at least $7$ boundary lattice points, we can assume by \ref{hollow_lattice_polygons} after a suitable affine unimodular transformation that we have $\inthull{P}\subseteq \R \times [0,1]$. Moreover, since $b(P)>4$, we cannot have points of $P$ in both $\R \times (-\infty,0)$ and $\R \times (1,\infty)$, and so, after a suitable affine unimodular transformation, we can also assume  that $P\subseteq \R \times [0,\infty)$, and so we get in particular $\interior{P}\cap\Z^2\in \Z \times \{1\}$. Since $i(P)>6$ we have $\text{length}(P\cap \R \times \{1\})>5$ and so a point of $P$ in $(\frac{5}{2},\infty)$ would imply interior lattice points of $P$ on $\{2\}\times \R$. So we have $P\subseteq \R \times [0,2]$. We can even go further to $P\subseteq \R \times [0,\frac{3}{2}]$, since a vertex on $\R\times \{0\}$ and on $\R\times \{2\}$ would imply that $\area{P}\leq 2i+2$ by \ref{half_integral_coordinates_for_big_areas}, which contradicts $\area{P}=i(P)+\frac{b(P)}{2}-1=2i+\frac{5}{2}$. Since $P$ is half-integral pseudo-integral, we have lattice points on every edge of $P$ by \ref{pseudo_crit} and so we have only one vertex of $P$ on $\R \times \{\frac{3}{2}\}$. After a suitable affine unimodular transformation we have that this vertex is $(0,\frac{3}{2})$ and $P\cap \R\times\{1\}\subseteq [0,i+1]\times\{1\}$, so we can assume all in all $P\subseteq \conv{(0,\frac{3}{2}), (0,0), (3i+3,0)}$.

\begin{figure}[H]
	\begin{tikzpicture}[x=0.6cm,y=0.6cm]
		\draw[step=2.0,black,thin,xshift=0cm,yshift=0cm] (-1,3.9) grid (23.9,-13);
		
		\draw [line width=1pt,color=black] (0,3)-- (0,0);
		\draw [line width=1pt,color=black] (0,0)-- (20,0);
		\draw [line width=1pt,color=black] (20,0)-- (16,1);
		\draw [line width=1pt,color=black] (16,1)-- (0,3);
		
		\draw [fill=black] (0,3) circle (2.5pt);
		\draw [fill=black] (16,1) circle (2.5pt);
		\draw [fill=black] (0,0) circle (2.5pt);
		\draw [fill=black] (2,0) circle (2.5pt);
		\draw [fill=black] (4,0) circle (2.5pt);
		\draw [fill=black] (6,0) circle (2.5pt);
		\draw [fill=black] (8,0) circle (2.5pt);
		\draw [fill=black] (10,0) circle (2.5pt);
		\draw [fill=black] (12,0) circle (2.5pt);
		\draw [fill=black] (14,0) circle (2.5pt);
		\draw [fill=black] (16,0) circle (2.5pt);
		\draw [fill=black] (18,0) circle (2.5pt);
		\draw [fill=black] (20,0) circle (2.5pt);
		\draw [fill=black] (0,2) circle (2.5pt);
		\draw [fill=black] (2,2) circle (2.5pt);
		\draw [fill=black] (4,2) circle (2.5pt);
		\draw [fill=black] (6,2) circle (2.5pt);
		\draw [fill=black] (8,2) circle (2.5pt);
		
		\draw[color=black] (19,1.4) node[fill=white] {$P_{2,3,13,a}$ };
		\draw[color=black] (22,1.4) node[fill=white] {$a=0$ };

		\draw [line width=1pt,color=black] (0,-1)-- (0,-3);
		\draw [line width=1pt,color=black] (0,-3)-- (2,-4);
		\draw [line width=1pt,color=black] (2,-4)-- (22,-4);
		\draw [line width=1pt,color=black] (22,-4)-- (16,-3);
		\draw [line width=1pt,color=black] (16,-3)-- (0,-1);
		
		\draw [fill=black] (0,-1) circle (2.5pt);
		\draw [fill=black] (16,-3) circle (2.5pt);
		\draw [fill=black] (0,-3) circle (2.5pt);
		\draw [fill=black] (22,-4) circle (2.5pt);
		\draw [fill=black] (2,-4) circle (2.5pt);
		\draw [fill=black] (4,-4) circle (2.5pt);
		\draw [fill=black] (6,-4) circle (2.5pt);
		\draw [fill=black] (8,-4) circle (2.5pt);
		\draw [fill=black] (10,-4) circle (2.5pt);
		\draw [fill=black] (12,-4) circle (2.5pt);
		\draw [fill=black] (14,-4) circle (2.5pt);
		\draw [fill=black] (16,-4) circle (2.5pt);
		\draw [fill=black] (18,-4) circle (2.5pt);
		\draw [fill=black] (20,-4) circle (2.5pt);
		\draw [fill=black] (0,-2) circle (2.5pt);
		\draw [fill=black] (2,-2) circle (2.5pt);
		\draw [fill=black] (4,-2) circle (2.5pt);
		\draw [fill=black] (6,-2) circle (2.5pt);
		\draw [fill=black] (8,-2) circle (2.5pt);
		
		\draw[color=black] (19,-2.6) node[fill=white] {$P_{2,3,13,a}$ };
		\draw[color=black] (22,-2.6) node[fill=white] {$a=1$ };

		\draw [line width=1pt,color=black] (0,-5)-- (0,-6);
		\draw [line width=1pt,color=black] (0,-6)-- (1,-8);
		\draw [line width=1pt,color=black] (1,-8)-- (22,-8);
		\draw [line width=1pt,color=black] (22,-8)-- (16,-7);
		\draw [line width=1pt,color=black] (16,-7)-- (0,-5);
		
		\draw [fill=black] (0,-5) circle (2.5pt);
		\draw [fill=black] (16,-7) circle (2.5pt);
		\draw [fill=black] (1,-8) circle (2.5pt);
		\draw [fill=black] (22,-8) circle (2.5pt);
		\draw [fill=black] (2,-8) circle (2.5pt);
		\draw [fill=black] (4,-8) circle (2.5pt);
		\draw [fill=black] (6,-8) circle (2.5pt);
		\draw [fill=black] (8,-8) circle (2.5pt);
		\draw [fill=black] (10,-8) circle (2.5pt);
		\draw [fill=black] (12,-8) circle (2.5pt);
		\draw [fill=black] (14,-8) circle (2.5pt);
		\draw [fill=black] (16,-8) circle (2.5pt);
		\draw [fill=black] (18,-8) circle (2.5pt);
		\draw [fill=black] (20,-8) circle (2.5pt);
		\draw [fill=black] (0,-6) circle (2.5pt);
		\draw [fill=black] (2,-6) circle (2.5pt);
		\draw [fill=black] (4,-6) circle (2.5pt);
		\draw [fill=black] (6,-6) circle (2.5pt);
		\draw [fill=black] (8,-6) circle (2.5pt);
		
		\draw[color=black] (19,-6.6) node[fill=white] {$P_{2,3,13,s}$ };

		\draw [line width=1pt,color=black] (0,-9)-- (0,-10);
		\draw [line width=1pt,color=black] (0,-10)-- (1,-12);
		\draw [line width=1pt,color=black] (1,-12)-- (23,-12);
		\draw [line width=1pt,color=black] (23,-12)-- (8,-10);
		\draw [line width=1pt,color=black] (8,-10)-- (0,-9);
		
		\draw [fill=black] (0,-9) circle (2.5pt);
		\draw [fill=black] (23,-12) circle (2.5pt);
		\draw [fill=black] (1,-12) circle (2.5pt);
		\draw [fill=black] (22,-12) circle (2.5pt);
		\draw [fill=black] (2,-12) circle (2.5pt);
		\draw [fill=black] (4,-12) circle (2.5pt);
		\draw [fill=black] (6,-12) circle (2.5pt);
		\draw [fill=black] (8,-12) circle (2.5pt);
		\draw [fill=black] (10,-12) circle (2.5pt);
		\draw [fill=black] (12,-12) circle (2.5pt);
		\draw [fill=black] (14,-12) circle (2.5pt);
		\draw [fill=black] (16,-12) circle (2.5pt);
		\draw [fill=black] (18,-12) circle (2.5pt);
		\draw [fill=black] (20,-12) circle (2.5pt);
		\draw [fill=black] (0,-10) circle (2.5pt);
		\draw [fill=black] (2,-10) circle (2.5pt);
		\draw [fill=black] (4,-10) circle (2.5pt);
		\draw [fill=black] (6,-10) circle (2.5pt);
		\draw [fill=black] (8,-10) circle (2.5pt);
		
		\draw[color=black, opacity=1] (19,-10.6) node[fill=white] {$P_{2,3,13,s_3}$ };
	\end{tikzpicture}
	\caption{Polygons, which are affine unimodular equivalent to the four pseudo-integral polygons with denominator $2$, $3$ interior lattice points and $13$ boundary lattice points, i.e. equivalent to $P_{2,3,13,0}$, $P_{2,3,13,1}$, $P_{2,3,13,s}$ and $P_{2,3,13,s_3}$. }
\end{figure}
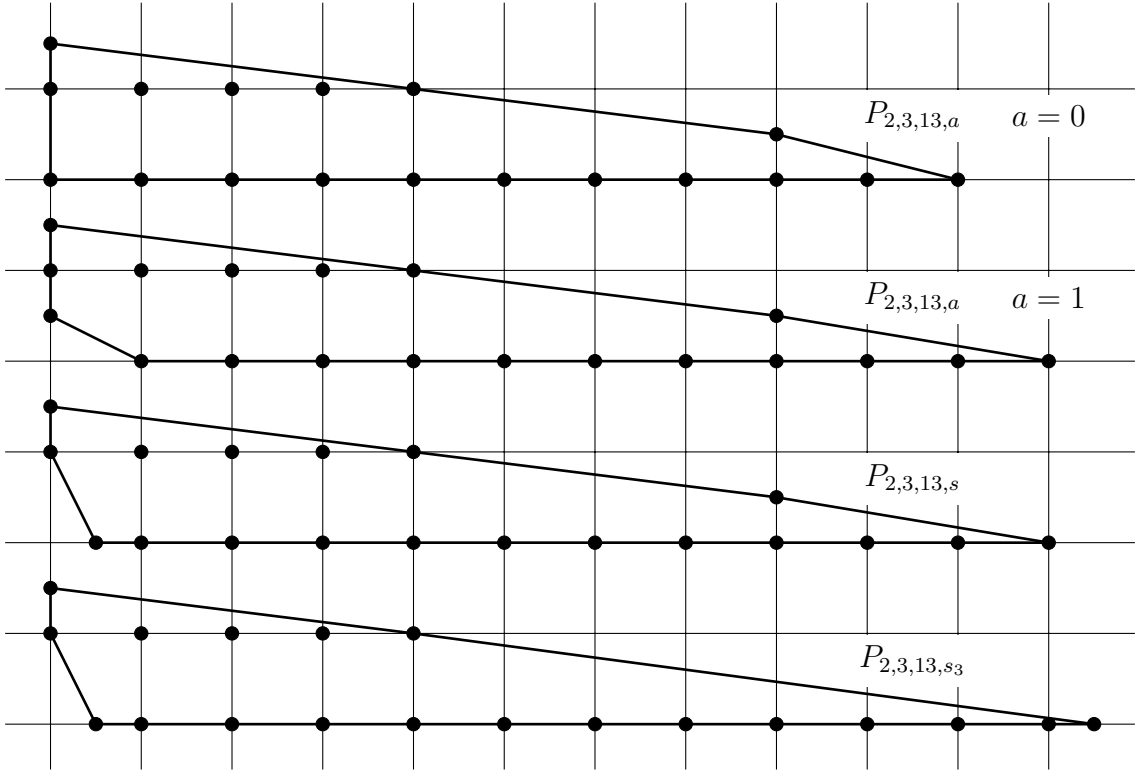

In the second step, we show that if $P$ is realized as subpolygon of the triangle $\conv{(0,\frac{3}{2}), (0,0), (3i+3,0)}$, then $(0,1), (i+1,1)\in P$. Since $P$ is pseudo-integral, we have by \ref{detecting_half_integral_pseudo_integral} for the number of half-integral interior points $i(2P)$ of $P$ that $i(2P)=6i+4$. This is the same number of half-integral interior points as in $\conv{(0,\frac{3}{2}), (0,0), (3i+3,0)}$ because
\begin{align*}
&|\interior{\conv{(0,3),(0,0),(6i+6,0)}}\cap \Z^2|\\=&|\{(l,1) \mid 1\leq l< 4i+4\}\cup \{(l,2) \mid 1\leq l< 2i+2\}|\\
=&6i+4.
\end{align*}
We also know from the pseudo-integrality of $P$ that there must be a lattice point between every two half-integral points on the boundary of $P$, which are not lattice points. Therefore, and to get $i(2P)=6i+4$, we must have $(0,1), (i+1,1)\in P$.

To get the complete description of $P$ in the last step, we distinguish whether $(0,\frac{1}{2})$ and $(2i+2,\frac{1}{2})$ are both points of $P$ or not.

We start with the case where $(0,\frac{1}{2})$ and $(2i+2,\frac{1}{2})$ are both points of $P$. Since $(0,\frac{1}{2})\in P$, the next half-integral point counterclockwise on the boundary of $P$ after $(0,\frac{1}{2})$ must be a lattice point $(a,0), a\in \Z_{\geq 0}$ because of the pseudointegrality. 

To get $b(2P)=2b(P)=4i+14$ the next vertex of $P$ counterclockwise after $(a,0)$ must be $(a+2i+4,0)$. So
\begin{align*}
P\cong P_{2,i,2i+7,a}:=\conv{\left(0,\frac{3}{2}\right), \left(0,\frac{1}{2}\right), (a,0), (a+2i+4,0), \left(2i+2,\frac{1}{2}\right)}.
\end{align*} 
To have $(2i+2,\frac{1}{2})$ really as a boundary point, we must additionally have the inequality $a+2i+4\leq 3i+3$, i.e. $a\leq i-1$. Moreover, we get affine unimodular equivalent polygons if and only if we take $a'=i-1-a$ instead of $a$. So we must restrict to $0\leq a\leq \frac{i-1}{2}$ to have each polygon only once.

In the case that $(0,\frac{1}{2})$ and $(2i+2,\frac{1}{2})$ are not both points of $P$, we can assume after a suitable affine unimodular transformation that $(0,\frac{1}{2})$ is not a point of $P$. To get $i(2P)=6i+4$, we must have $(\frac{1}{2},\frac{1}{2})$ as an interior point of $P$, and so the next half-integral point on the boundary of $P$ counterclockwise after $(0,1)$ must be $(\frac{1}{2},0)$. To get $b(2P)=2b(P)=4i+14$, the next vertex counterclockwise from $P$ after $(\frac{1}{2},0)$ must be $(2i+5,0)$ if $(2i+2,\frac{1}{2})\in P$ and $(2i+\frac{11}{2},0)$ if $(2i+2,\frac{1}{2})\notin P$. So we have
\begin{align*}
P\cong P_{2,(i,2i+7),s}=\conv{\left(0,\frac{3}{2}\right), (0,1), \left(\frac{1}{2},0\right), (2i+5,0), \left(2i+2,\frac{1}{2}\right)}  
\end{align*}
or
\begin{align*}
P\cong \conv{\left(0,\frac{3}{2}\right), (0,1), \left(\frac{1}{2},0\right), \left(2i+\frac{11}{2},0\right), (i+1,1)}  .  
\end{align*}
In the first case we see that $(2i+2,\frac{1}{2})$ is a boundary point only if $2i+5\leq 3i+3$, i.e. $i\geq 2$, so we have this additional constraint. In the second case, we have $(\frac{3}{2}i+\frac{13}{4},\frac{1}{2})$ as a boundary point, so we get $(2i+\frac{3}{2},\frac{1}{2})\in P$, $(2i+2,\frac{1}{2})\notin P$ if and only if $\frac{3}{2}i+\frac{13}{4}=2i+\frac{7}{4}$, i.e. $i=3$.\\

Conversely, we see that all given polygons have the correct number of interior and boundary lattice points by construction and the correct area to satisfy Pick's formula. Moreover, they all have denominator $2$ and lattice points on all edges, so we conclude that they are all pseudo-integral with denominator $2$ and have exactly $2i(P)+7$ boundary lattice points.
\end{proof}

\begin{coro}
For $i\in \Z_{\geq 1}$ there are exactly $\lfloor\frac{i-1}{2}\rfloor +2+\delta(i-3)-\delta(i-1)$ pseudo-integral polygons $P\subseteq \R^2$ with $\denom{P}=2$, $i(P)=i$ and $b(P)=2i(P)+7$, where 
\begin{align*}
	\delta\colon \Z\to \Z, k \mapsto \begin{cases}
		1 & \text{ for } k=0\\
		0 & \text{ else}
	\end{cases}.
\end{align*}
\end{coro}

\vspace{15mm}

\textbf{Acknowledgements.} I would like to thank Katharina Kir\'{a}ly for her helpful comments. I would also like to thank Tyrrell B. McAllister for his interest in my work and our discussions on pseudo-integral polygons.

\end{document}